\documentclass[11pt,reqno]{amsproc}

\usepackage[sc]{mathpazo}\renewcommand{\mathbf}{\mathbold}
\normalfont
\usepackage[T1]{fontenc}
\usepackage[ttdefault=true]{AnonymousPro}

\usepackage[margin=0.9in]{geometry}
\usepackage{comment}
\usepackage{scalerel,nicefrac}
\usepackage[all,cmtip]{xy}
\usepackage{longtable}
\usepackage{dirtytalk}
\usepackage[dvipsnames]{xcolor}

\usepackage{rotating, setspace}

\makeatletter
\def\l@subsection{\@tocline{2}{0pt}{2.5pc}{5pc}{}}
\def\l@subsubsection{\@tocline{2}{0pt}{5pc}{7.5pc}{}}

\makeatother

\usepackage{bbm}
\usepackage{listings}
\usepackage{mathtools,amssymb}
\usepackage{bookmark}
\usepackage{tikz-cd}
\usepackage{subfigure}
\usepackage{shuffle}
\usepackage{mathrsfs}
\usepackage{multicol}
\usepackage{booktabs}
\usepackage{makecell}
\usepackage{mleftright}
\mleftright

\makeatletter
\renewcommand{\@secnumfont}{\bfseries}
\makeatother

\hypersetup{
    colorlinks=true,
    linkcolor=BrickRed,
    citecolor=RoyalBlue,
    filecolor=blue,      
    }

\numberwithin{equation}{section}
\usepackage{cite}

\newcommand{\ps}[1]{\mkern-.25mu\mathbin{\left(\mkern-3.5mu\left({#1}\right)\mkern-3.5mu\right)}}

\newcommand{\N}{\mathbb{N}}

\newcommand{\R}{\mathbb{R}}
\newcommand{\C}{\mathbb{C}}

\newcommand{\E}{\mathbb{E}}

\newcommand{\ba}{\mathbf{a}}
\newcommand{\bb}{\mathbf{b}}
\newcommand{\bp}{\mathbf{p}}

\newcommand{\bc}{\mathbf{c}}
\newcommand{\bG}{\mathbf{G}}
\newcommand{\bH}{\mathbf{H}}
\newcommand{\bw}{\mathbf{w}}
\newcommand{\bq}{\mathbf{q}}
\newcommand{\balpha}{{\boldsymbol{\alpha}}}

\newcommand{\bzero}{\mathbf{0}}
\newcommand{\bX}{{\boldsymbol{X}}}
\newcommand{\dif}{\mathrm{d}}

\newcommand{\tbp}{\widetilde{\bp}}
\newcommand{\tmu}{\widetilde{\mu}}

\newcommand{\tG}{\widetilde{G}}

\newcommand{\hW}{\widehat{W}}
\newcommand{\hL}{\widehat{L}}

\newcommand{\cH}{\mathcal{H}}
\newcommand{\cN}{\mathcal{N}}
\newcommand{\cL}{\mathcal{L}}
\newcommand{\cM}{\mathcal{M}}
\newcommand{\cB}{\mathcal{B}}
\newcommand{\cA}{\mathcal{A}}

\newcommand{\cD}{\mathcal{D}}

\newcommand{\fM}{\mathfrak{M}}

\newcommand{\Sh}{\text{Sh}}
\newcommand{\tr}{\intercal}
\newcommand{\Mat}{\operatorname{Mat}}
\newcommand{\Sym}{{\operatorname{Sym}}}
\newcommand{\rank}{\operatorname{rank}}
\newcommand{\im}{\operatorname{im}}
\newcommand{\andd}{\quad \text{and} \quad}
\newcommand*{\pmat}[1]{\begin{pmatrix}#1\end{pmatrix}}

\newcommand{\qshuffle}{\mathbin{\widehat{\shuffle}}}
\def\zero{\texttt{0}}
\newcommand{\one}{\mathbbm{1}}
\newcommand{\tone}{\texttt{1}}

\newcommand{\sbas}{w}

\newcommand{\sbasvec}{\bw}

\newcommand{\gvec}{W}

\newcommand{\tdeg}{\text{tdeg}}
\newcommand{\sdeg}{\text{sdeg}}
\newcommand{\gr}{{\operatorname{gr}}}

\newcommand{\bone}{\mathbf{1}}

\newcommand{\eps}{\varepsilon}
\newcommand{\mfu}{\mathfrak{u}}
\newcommand{\scal}[1]{\langle #1 \rangle}
\newcommand{\mrd}{\mathop{}\!\mathrm{d}}

\usepackage[thmtools-compat]{keytheorems}
\declaretheorem[numberwithin=section, style=remark]{remark}
\declaretheorem[sibling=remark]{conjecture, proposition, theorem,%
lemma, corollary}

\declaretheorem[sibling=remark, style=definition]{definition, example}

\usepackage[nameinlink]{cleveref}
\AddToHook{cmd/appendix/before}{\crefalias{section}{appendix}}

\begin{document}
\title{Orthogonal polynomials on path-space}
\date{February 21, 2026}

\author{Ilya Chevyrev}
\address{SISSA (International School for Advanced Studies), via Bonomea 265, 34136 Trieste, Italy}
\email{ichevyrev@gmail.com}

\author{Emilio Ferrucci}
\address{SISSA (International School for Advanced Studies), via Bonomea 265, 34136 Trieste, Italy}
\email{emilio.ferrucci@sissa.it}

\author{Darrick Lee}
\address{School of Mathematics and Maxwell Institute, University of Edinburgh, Edinburgh EH9 3FD, Scotland}
\email{darrick.lee@ed.ac.uk}

\author{\\ Terry Lyons}
\address{Mathematical Institute, University of Oxford. Woodstock Rd, Oxford OX2 6GG UK.}
\address{Department of Mathematics,
Imperial College London,
180 Queen's Gate,
London, SW7 2AZ
UK}

\email{terry.lyons@maths.ox.ac.uk}

\author{Harald Oberhauser}
\address{Mathematical Institute, University of Oxford. Woodstock Rd, Oxford OX2 6GG UK.}
\email{oberhauser@maths.ox.ac.uk}

\author{Nikolas Tapia}
\address{Weierstrass Institute, Anton-Wilhelm-Amo-Str. 39, 10117, Berlin, Germany.}
\address{Institut für Mathematik, Humboldt-Universität zu Berlin, Rudower Chaussee 25, 12489 Berlin, Germany.}
\email{tapia@wias-berlin.de}

\begin{abstract}
We consider the orthogonalisation of the signature of a stochastic process as the analogue of orthogonal polynomials on path-space. Under an infinite radius of convergence assumption, we prove density of linear functions on the signature in $L^p$ functions on grouplike elements, making it possible to represent a square-integrable function on (rough) paths as an $L^2$-convergent series. By viewing the shuffle algebra as commutative polynomials on the free Lie algebra, we revisit much of the theory of classical orthogonal polynomials in several variables, such as the recurrence relation and Favard's theorem. Finally, we restrict our attention to the case of Brownian motion with and without drift, and prove that dimension-independent orthogonal signature exists with drift but not without. We end with numerical examples of how orthogonal signature polynomials of Brownian motion can be applied for the approximation of functions on paths sampled from the Wiener measure.
\end{abstract}

\maketitle
\vspace{-10pt}
\begin{center}
\begin{small}
    February 21, 2026
\end{small}
\end{center}

{\footnotesize
\tableofcontents 
}

\section{Introduction}\label{sec:intro}
Orthogonal polynomials in one or several variables have applications in numerical analysis \cite{trefethen},
mathematical physics \cite{MQM}, probability theory and stochastic processes \cite{xiu, schoutens}, mathematical finance
\cite{MR4067070}, and in many other areas of mathematics and the applied sciences. One simple example that exhibits
their potential is the following. Take a function $f \colon [-1,1] \to \mathbb R$. If $f$ is analytic at $0$, we may
approximate it with its Taylor polynomials near $0$. Outside its radius of convergence, however, the Taylor
approximation will not converge, and even within, it may converge slowly. On the other hand, as long as $f \in
L^2[-1,1]$, its $L^2$-projections onto the space spanned by polynomials of degree $N$, $\Pi_N f$, will converge to $f$ in $L^2$
as $N \to \infty$ (and even uniformly if $f$ has H\"older regularity greater than $\frac 12$ \cite{sax}).
These projections can be found computing the $L^2$ inner product $(\ell_n, f) = \int_{-1}^1 f(x) \ell_n(x) \dif x$ with the Legendre polynomials, orthogonal w.r.t.\ the Lebesgue measure on $[-1,1]$ and expanding 
\begin{equation}\label{eq:legendre}
f = \sum_{n = 0}^\infty \frac{(\ell_n, f)}{(\ell_n, \ell_n)} \ell_n.
\end{equation}
The series truncated at $N$ coincides with $\Pi_Nf$. A comparison of the two types of approximation is given in \Cref{fig:approximations}.
\begin{figure}[htbp]
    \centering
    \includegraphics[width=\textwidth]{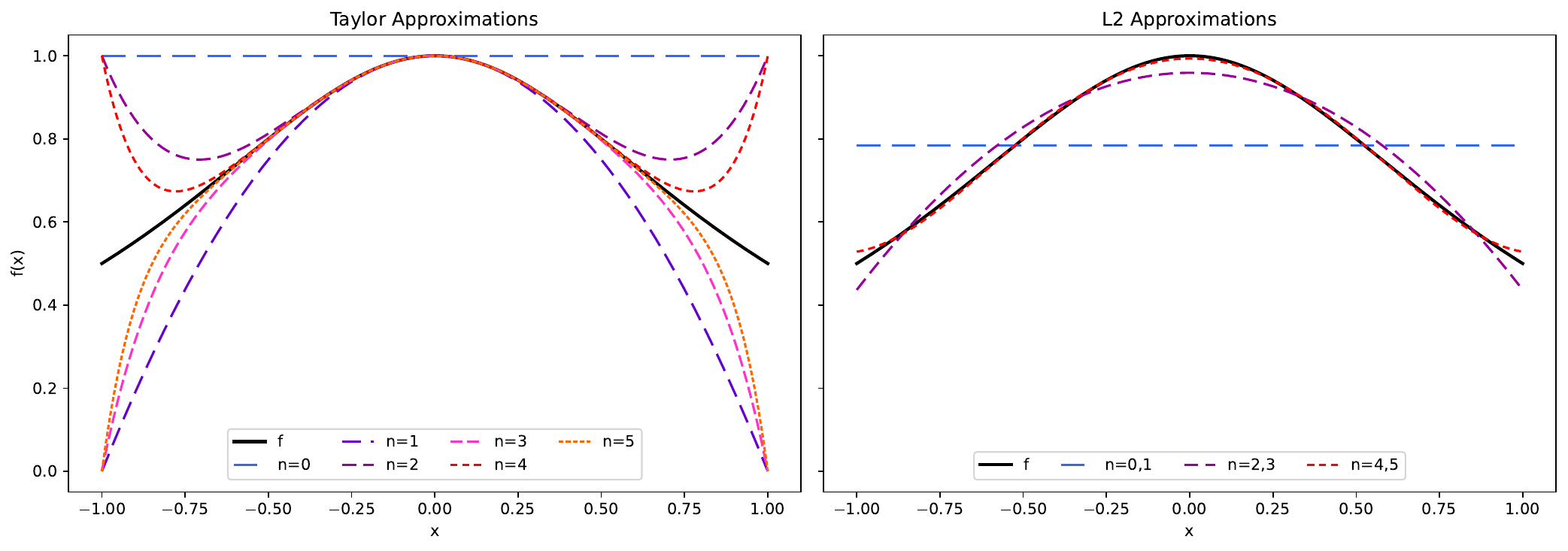}
    \vspace{-4ex}
    \caption{Taylor and $L^2$ approximations (computed with Legendre orthogonal polynomials) of $f(x) = \frac{1}{1 + x^2}$. Since $f$ has poles at $\pm i$, the Taylor approximations do not converge uniformly on $(-1,1)$.}
    \label{fig:approximations}
\end{figure}

Since the work of Chen \cite{chen_iterated_1954}, iterated path integrals
\[
S^n(X)_{0,T} \coloneqq \int_{0 < t_1 < \ldots < t_n < T} \dif X_{t_1} \otimes \cdots \otimes \dif X_{t_n}, \qquad X \colon [0,T] \to \mathbb R^d
\]
have established themselves as the path-space analogue of polynomials. For rougher signals they are ill-defined from analysis alone, but in some cases can be defined probabilistically; specifying the first few subject to certain requirements forms the definition of a
\emph{rough path} \cite{Lyo98}, i.e., a specification of a theory of differential equations. The full stack of iterated integrals, its \emph{signature} $S(X)_{0,T}$, has been shown to characterise the path up to \emph{tree-like equivalence}, a generalised form of reparametrisation which allows \say{retracings}, see \cite{chen_integration_1958, HL10, siguniqueness} for the proof of this statement in the case of $X$ piecewise regular, bounded variation, and rough, respectively. Moreover, since coordinate functions on the signature form an algebra under the shuffle product (see \eqref{eq:shuffle} below), the Stone-Weierstrass theorem guarantees that, roughly speaking, for any continuous function on a compact set of (rough) paths $K$ and any $\varepsilon > 0$ there exists a truncation level $N$ and a linear function $\lambda_N$ of $S_N(X)_{0,T}$ (the signature truncated at degree $N$, see \eqref{eq:sig} below), such that 
\begin{equation}\label{eq:SW}
\sup_{X \in K}|F(X) - \langle \lambda_N, S_N(X)_{0,T} \rangle| < \varepsilon.
\end{equation}
This result proves that functions on paths can be approximately linearised in terms of the signature, but has a couple of shortcomings. First of all, most sets of paths on which interesting functions are defined (e.g.\ the support of the law on path-space of the solution to an SDE \cite{SV72}, or indeed the support of the law of Brownian motion itself) are not compact. Secondly, writing $\lambda_N = \sum_{k = 0}^N \lambda_N^k$ with $\lambda_N^k$ acting on tensor degree $k$, it cannot be expected that a linear function $\lambda_{N'}$ ($N' > N$) that achieves a better approximation ($\varepsilon' < \varepsilon$) should have the same low-order tensor projections: in general $\lambda_N^k \neq \lambda_{N'}^k$ for $k \leq N$, i.e.\ the approximation is not stable under increasing the precision of the approximation.
This latter phenomenon is not specific to approximations of the type \eqref{eq:SW} for signatures and even occurs for $L^2$-approximations of functions in one variable with polynomials: if $f \in L^2[-1,1]$ as above, writing $(\Pi_N f)(x) = \sum_{k = 0}^N \lambda^k_N x^k$, the same instability occurs, simply because monomials are correlated w.r.t.\ the Lebesgue measure on $[-1,1]$ (or indeed w.r.t.\ any measure that is not the Dirac delta $\delta_0$).

Iterated integrals also naturally appear in numerical schemes for stochastic differential equations (SDEs) \cite{KP10}, see also \cite{LO14, dupire2023functional} for similar schemes which deal with path-dependency. These approximations are quite different to \eqref{eq:SW}, since they generally rely on the interval $[0,T]$ being split into many sub-intervals, and on the expansion to be performed on each interval, for convergence to be guaranteed. More specifically, an important example of function on paths is $\varphi(Y_T)$, where $\varphi$ is a smooth function and $Y$ is the solution to the controlled differential equation with smooth vector fields $F_1,\ldots, F_d \in C^\infty(\mathbb R^n, \mathbb R^n)$:
\begin{equation}\label{eq:RDEs}
\dif Y_t = F(Y_t) \dif \bX_t, \qquad Y_0 = y_0 .
\end{equation}
Applying the chain rule $N$ times yields the expansion
\begin{equation}\label{eq:taylor}
\begin{split}
\varphi(Y_T) &= \varphi(y_0) + \sum_{n = 1}^N \langle F_{\alpha_1} \cdots F_{\alpha_n}\varphi(y_0), S(X)_{0,T}^{\alpha_1\ldots\alpha_n} \rangle + R_N(F, \varphi, X)_{0,T} \\
\text{with}\quad R_N(F, \varphi, \bX)_{0,T} &= \int_{0 < t_1 < \ldots < t_N < T} F_{\alpha_1} \cdots F_{\alpha_N}\varphi(Y_{t_1}) \dif X_{t_1}^{\alpha_1} \cdots \dif X_{t_N}^{\alpha_N} .
\end{split}
\end{equation}
Failures of analyticity such as \Cref{fig:approximations}, which a fortiori occur in the path setting (since iterated
integrals embed monomials) show that this error cannot generally be made small without controlling the size of the interval. This type of approximation scheme therefore has the benefit of having explicit identities for the coefficients in the expansion, but the drawbacks of requiring a more explicit form of the function on paths (e.g.\ the vector fields in the differential equation) and of requiring iteration for convergence.

Briefly returning to a polynomial $L^2$-approximations for functions in one variable, observe that a necessary condition for a representation like \eqref{eq:legendre} to hold is that polynomials be dense in $L^2$. This is always true if the measure is compactly supported, but for general measures on the real line the question of density of polynomials in $L^p$ is intimately related to that of moment-determinacy \cite{BC81}, i.e.\ whether the measure is unique given its moments $\mathbb E_\mu x^n$. On path-space, the analogous question asks whether the \emph{expected signature} $\mathbb E S(X)_{0,T}$ of a stochastic process determines its law. The question was answered in the affirmative in \cite{characteristic} under an \say{infinite radius of convergence} assumption, which holds true for many stochastic processes lifted to rough paths considered in the literature \cite{FV10}.

The goal of this paper is to combine orthogonal polynomials with path signatures, forming what we consider the analogue
of orthogonal polynomials on path-space. Our contributions are as follows. In \Cref{sec:L2} we briefly introduce the signature of a stochastic process and show how its canonical coordinates can be (block-)orthogonalised with the familiar Gram-Schmidt procedure, given the knowledge of the expected signature of the process. Next, we prove that, under the same infinite radius of convergence assumption as in \cite{characteristic},
linear functions on the signature are dense in $L^p$ functions on paths (modulo tree-like equivalence) for any $p\in [1,\infty)$.
While this result holds under the same hypotheses as in \cite{characteristic}, it does not follow directly from the latter.
Combined with orthogonalisation, our density result yields an $L^2$-series expansion for a square integrable function on paths $F$. This circumvents both the issues of poor convergence away from the point of expansion of \eqref{eq:taylor} and the issues of stability under increase in precision of \eqref{eq:SW}.
We remark that $L^p$-density of signature functionals was also very recently proved in \cite{lpdensity} using weighted spaces for a class of time-augmented processes. Our density result is, however, stronger\footnote{More specifically,
\cite{lpdensity} requires time-augmentation and integrability of the weight $w(\bX) = \exp(\beta\|\bX\|_{\alpha}^N)$ for some $\beta>0$,
where $N=\lfloor 1/\alpha \rfloor$ and $\|\bX\|_{\alpha}$ is the $\alpha$-H\"older homogeneous rough path norm.
Integrability of $w$ (for any $\beta>0$) implies an infinite radius of convergence of $\E S(\bX)_{0,T}$,
so our \Cref{thm:density} implies \cite[Thm.~3.4]{lpdensity}, but not conversely since we do not require time augmentation and do not assume integrability of $w$.
For example, \Cref{thm:density} applies to fractional Brownian motion with $1/4<H\leq1/3$ while \cite[Thm.~3.4]{lpdensity} does not because $w$ is not integrability as in this case $N=3$.
Moreover, the restriction to time augmented processes seems important in \cite{lpdensity} since the proof relies on
the uniform density result in weighted spaces from \cite{cuchiero_weighted}, which appears to use time augmentation in a crucial way.} and applies without any assumptions beyond an infinite radius of convergence.
Our proof is moreover very different from \cite{lpdensity} and follows more the style of the proof of moment determinacy \cite{characteristic}.

In \Cref{sec:orthPol} we begin by observing the known fact that a change of basis transforms the shuffle algebra, which governs the product of coordinate functions on the signature, into the polynomial algebra over an infinite-dimensional but graded space (the free Lie algebra). This enables us to revisit much of the theory of classical orthogonal polynomials in several variables \cite{dunkl_orthogonal_2014}, with the added difficulty of infinite dimensionality and anisotropic grading on the base variables. We prove versions of the celebrated three-term recurrence relation, which is no longer three-term but nevertheless yields a way of computing the orthogonal signature polynomials alternatively to Gram-Schmidt. We also prove a version of Favard's theorem, which characterises which sequences of polynomials arise as orthogonal polynomials with respect to an inner product. We then study when such inner products arise from probability measures on both the free Lie algebra and path space. 

In \Cref{sec:brownian} we specify our study to what is arguably the most important example of stochastic process, Brownian motion, seeking more explicit descriptions of the orthogonal signature polynomials. We ask the question of \emph{naturality}, i.e.\ roughly speaking whether the orthogonal polynomials associated to a $V$-valued brownian motion can be expressed without reference to a basis or even to the dimension of $V$. While this is true of ordinary Hermite polynomials in several variables, we prove, using symbolic linear algebra code that it is not true for the signature of Brownian motion, unless time is included as a coordinate. This concern is not merely motivated by aesthetics, rather, it speeds up the computation of the orthogonal signature compared to naively performing Gram-Schmidt orthogonalisation. We proceed with the explicit description and computation of orthogonal signature polynomials of time-augmented Brownian motion which are natural (basis- and dimension- free). We also give an independent, short, Wiener chaos-based proof of density of linear functions on the signature in $L^2$ in the time-augmented case. We end with a description of our implementation \cite{orthsig} of these orthogonal polynomials on Wiener space and show how it can be used in approximation problems.

We believe that, in future work, it would be interesting to consider the orthogonal signature associated to the time-augmented stochastic processes (many of which with jumps) associated to the Askey classification \cite{AA85}, see \cite{schoutens}. \medskip

{\footnotesize
\textbf{Acknowledgements.} We thank Cris Salvi for many helpful discussions on the topics of this paper.

IC and EF gratefully acknowledge support from the ERC via the Starting Grant SQGT 101116964.
During the first phase of this project, EF was employed at the University of Oxford and supported by the EPSRC programme grant [EP/S026347/1].
DL was supported by the Hong Kong Innovation and Technology Commission (InnoHK Project CIMDA) during part of this work. 
TL was supported in part by UK Research and Innovation (UKRI) through the Engineering and Physical Sciences Research Council (EPSRC) via Programme Grants [Grant No. UKRI1010: High order mathematical and computational infrastructure for streamed data that enhance contemporary generative and large language models], [Grant No. EP/S026347/1: Unparameterised multi-model data, high order signatures and the mathematics of data science], and the UKRI AI for Science award [Grant No. UKRI2385: Creating Foundational Benchmarks for AI in Physical and Biological Complexity]. He was also supported by The Alan Turing Institute under the Defence and Security Programme (funded by the UK Government) and through the provision of research facilities; by the UK Government; and through CIMDA@Oxford, part of the AIR@InnoHK initiative funded by the Innovation and Technology Commission, HKSAR Government.
HO was supported by [EP/Y028872/1] (An "Erlangen Programme" for AI) and by the Hong Kong Innovation and Technology Commission (InnoHK Project CIMDA).
NT acknowledges funding by the Deutsche Forschungsgemeinschaft (DFG, German Research Foundation) – CRC/TRR 388 ``Rough
Analysis, Stochastic Dynamics and Related Fields'' – Project ID 516748464.

For the purpose of open access, the authors have applied a Creative Commons Attribution (CC BY) license to any author accepted manuscript version arising from this submission.
}

\section{\texorpdfstring{$L^2$}{L²} orthogonal signature expansions}\label{sec:L2}

\subsection{Signature of a stochastic process}


Let $X$ be a continuous stochastic process in a time variable on $[0,T]$ taking values in the finite-dimensional vector space $V$, defined on a probability space $(\Omega, \mathcal F, \mathbb P)$.
Let $X$ be of bounded $p$-variation and let $\bX$ be a $p$-geometric rough path lift of $X$, which we assume to be $\mathcal F$-measurable (this is the case for all the main examples of stochastic rough paths). For a finite-dimensional vector space $U$ we denote
\begin{align}
    T_N(U) \coloneqq \bigoplus_{m=0}^N U^{\otimes m}, \qquad T(U) \coloneqq \bigoplus_{m=0}^\infty U^{\otimes m}, \qquad T\ps{U} \coloneqq \prod_{m=0}^\infty U^{\otimes m}.
\end{align}
We recall that, given the rough path lift $\bX$, its \emph{signature} is canonically defined \cite{Lyo98}, and represents the iterated integrals
\begin{equation}\label{eq:sig}
\begin{split}
S^n(\bX)_{0,T} &= \int_{0 < t_1 \ldots < t_n < T} \dif X_{t_1} \otimes \cdots \otimes \dif X_{t_n} \in V^{\otimes n} \\
S_N(\bX)_{0,T} &= (S^0(\bX)_{0,T}, \ldots, S^N(\bX)_{0,T}) \in T_N(V) \\
S(\bX)_{0,T} &= (S^0(\bX)_{0,T}, \ldots, S^N(\bX)_{0,T},\ldots) \in T\ps{V}
\end{split}
\end{equation}
defined according to the integration theory specified by $\bX$. We write $\langle \, \cdot \, , \, \cdot \, \rangle$ for the canonical dual pairing between $T(V^*)$ and $T\ps{V}$, e.g.\ for a word $w = \alpha_1 \ldots \alpha_n$ representing an elementary tensor in $(V^*)^{\otimes n}$
we have
\[
\langle w , S(\bX)_{0,T} \rangle \coloneqq \int_{0 < t_1 \ldots < t_n < T} \dif X_{t_1}^{\alpha_1} \otimes \cdots \otimes \dif X_{t_n}^{\alpha_n}.
\]
We write $|w| = n$ for the tensor degree of $w$. We recall that signature coordinates satisfy the \emph{shuffle identity}, i.e.\ for $u, v \in T(V^*)$
\begin{equation}\label{eq:shuffle}
\langle u, S(\bX) \rangle\langle v, S(\bX) \rangle = \langle u \shuffle v, S(\bX) \rangle.
\end{equation}
where $\shuffle$ is the shuffle product of the two words. For this reason we sometimes will call $\mathrm{Sh}(V^*) = T(V^*)$ when viewed as an algebra under the shuffle product
and correspondingly write $\mathrm{Sh}_N(V^*) = T_N(V^*)$.
Elements of $T\ps{V}$ satisfying \eqref{eq:shuffle} form a (formal) group $G(V) \subset T\ps{V}$, whose Lie algebra is the space of free Lie series~\cite{reutenauer_free_1993}.
Recall that $A = (A^0,A^1,\ldots) \in T\ps{V}$ has \emph{an infinite radius of convergence} if
\begin{equation}\label{eq:infiniteRadius}
\|A\|_\lambda := \sum_{n = 0}^\infty \|A^n\|_{V^{\otimes n}} \lambda^n < \infty \quad \text{for all } \lambda > 0,
\end{equation}
where we equip $V$ with a norm and $V^{\otimes n}$ with an arbitrary cross norm, e.g. the projective norm (by finite dimensionality of $V$, condition \eqref{eq:infiniteRadius} does not depend on the choice of norms).
Denote $\widetilde T\ps{V}$ the space of these elements and
\[
\widetilde G(V) \coloneqq \widetilde T\ps{V} \cap G(V).
\]
We equip $\widetilde T\ps{V}$ with the (locally convex) topology given by the norms $\|\cdot\|_\lambda$,
and $\tG(V)$ inherits this topology, making it a topological group, and a Polish (completely metrisable and separable) space~\cite{characteristic}.

It is a fact that $S(\bX)_{0,T}$ is an element of $\tG(V)$ for any rough path $\bX$, see e.g. \cite{LCL}.
Therefore, we may view the stochastic process $X$ as defining a probability measure on $\tG(V)$. From now on, unless there is an ambiguity as to the rough path lift $\bX$, we will drop the bold font and simply write $S(X)_{0,T}$ and similar; we will always use $X$ to denote the trace of $\bX$, i.e.\ its projection onto $V$.

The signature encodes almost all the information in the (rough) path, namely up to tree-like equivalence \cite{chen_integration_1958, HL10, siguniqueness}, a generalised form of reparametrisation that includes \say{retracings}. A common way of ensuring the path is tree-reduced, and thus that the full path can be recovered from the signature, is to include time as a coordinate, see e.g.\ \cite{fermanian}.  We let $\mathcal G$ denote the sigma-algebra generated by $S(X)_{0,T}$. 
\begin{remark}[$\mathcal G = \mathcal F$]\label{rem:FG}
If time is included as zero-th coordinate of $X$, it follows from the fact that $\bX$ is tree-reduced (i.e.\ $t \mapsto \bX_{0,t}$ is injective) and measurability of the map that reconstructs the tree-reduced representative from the signature \cite{geng}, that $\mathcal G = \mathcal F$. In many cases of interest, such as Brownian motion \cite{treelikeBm} and the law of certain multidimensional diffusions \cite{treelikeDiffusion}, $\mathcal G = \mathcal F$ even without including time as a coordinate of the process. 
\end{remark}

\subsection{Orthogonalisation of signature features}

We assume that $\langle w, S(X)_{0,T}\rangle \in L^2(\mathcal G)$ for all $w \in \Sh(V^*)$. Define the positive semi-definite symmetric bilinear form
\begin{equation}\label{eq:innerProd}
(u, v) \coloneqq \mathbb E \langle u, S(X)_{0,T} \rangle \langle v,  S(X)_{0,T} \rangle = \langle u \shuffle v , \mathbb E S(X)_{0,T} \rangle .
\end{equation}
This bilinear form is, in general, only positive semi-definite because the natural map $\phi \colon \Sh_N(V^*) \to L^2(\mathcal G)$ is not in general an inclusion. Let $N$ be its nullspace and $W$ a choice of a direct complement to it, so that $\Sh(V^*) = N \oplus W$ and $\phi$ embeds $W$ into $L^2(\mathcal G)$. We may always choose $W$ to be of the form $W = \bigoplus_{n = 0}^\infty W^n$ with $W^n \subset (V^*)^{\otimes n}$:
indeed, let $\mathcal W^n$ be the pre-image through the quotient $\pi \colon \Sh(V^*) \twoheadrightarrow \Sh(V^*)/N$ of some basis of $\pi((V^*)^{\otimes n})$ and take $W^n = \mathrm{span}(\mathcal W^n)$. As usual call $\mathcal W_n = \bigsqcup_{k = 0}^n \mathcal W^k$ the basis of $W_n = \bigoplus_{k = 0}^n W^k$. Let
\begin{equation}\label{eq:PiN}
\begin{split}
\Pi_N \colon L^2(\mathcal G) &\twoheadrightarrow \mathrm{span}\{ \langle \ell, S(X)_{0,T} \rangle \mid \ell \in
\mathrm{Sh}_N(V^*) \} \\
\Pi^n \colon L^2(\mathcal G) &\twoheadrightarrow \mathrm{span}\{ \langle \ell, S(X)_{0,T} \rangle \mid \ell \in (V^*)^{\otimes n} \} \\
\mathbbm 1 - \Pi_N \eqqcolon \Pi_N^\bot \colon L^2(\mathcal G) &\twoheadrightarrow \mathrm{span}\{ \langle \ell,
S(X)_{0,T} \rangle \mid \ell \in \mathrm{Sh}_N(V^*) \}^\bot
\end{split}
\end{equation}
denote orthogonal projections, where $\bot$ denotes the orthogonal complement in $L^2(\mathcal G)$. We consider the \emph{block-orthogonalisation map} $p$
\begin{equation}\label{eq:p}
p \colon W \xrightarrow{\cong} W, \qquad p|_{W^n} = \Pi_{n-1}^\bot|_{W^n}, \qquad w \mapsto p_w.
\end{equation}
Notice that $p|_{W^n}$ takes values in $W_{n-1} \oplus W^n = W_n$ and therefore so does $p|_{W_n}$, i.e.\ $p$ is triangular, and moreover monic, namely $\Pi^n \circ p|_{W^n} = \mathbbm 1$. It can be computed by solving, for $w \in \mathcal W^n$ (and extended linearly to $W^n$)
\begin{equation}
p_w = w - \sum_{u \in \mathcal W_{n-1}} \lambda_u u, \qquad \Big( w - \sum_{u \in \mathcal W_{n-1}} \lambda_u u, v \Big) = 0 \quad \text{for } v \in \mathcal W_{n-1} .
\end{equation}
Let $\mathcal V$ be a basis of $V^*$, then the words $\alpha_1\ldots \alpha_n$ in $\mathcal V^n$ define a basis of $\mathrm{Sh}(V^*)$, and write $\mathcal V^\bullet \coloneqq \bigsqcup_{n \in \mathbb N}(\mathcal V^*)^{\otimes n}$. If the bilinear form
$(\, \cdot \, , \, \cdot \,)$ is positive-definite, then $W = \Sh(V^*)$; in many cases it is possible to realise $W^n$ as the linear span of a subset of $\mathcal V^\bullet$ (see \Cref{subsec:ito} below for an explicit example of this). Assume this is the case, and that  $\mathcal W^n$ is a subset of words in $\mathcal V^n$. We then call the elements $\{p_w\}_{w \in \mathcal W}$ the \emph{block-orthogonal shuffle polynomials} since $(p_v, p_w) = 0$ for $v, w \in \mathcal W$ with $|v| \neq |w|$.

Assume furthermore the basis $\mathcal W$ is totally ordered in a way that respects the grading ($v < w$ if $|v| < |w|$). We may then define a fully orthogonal basis by the Gram-Schmidt orthogonalisation procedure:
\begin{equation}\label{eq:ptilde}
\widetilde p_w \coloneqq w - \sum_{\substack{v \in \mathcal W^\bullet  \\ v < w \\ (\widetilde p_v,\widetilde p_v) > 0}} \frac{(w,\widetilde p_v)}{(\widetilde p_v,\widetilde p_v)}
\widetilde p_v
\end{equation}
where $v < w$ if $|v| < |w|$ or if $|v| = |w|$ and $v < w$ lexicographically. Then $\{\widetilde p_w\}_{w \in \mathcal W^n}$ is an orthogonal basis of $\mathrm{span}\{p_w \mid w \in \mathcal W^n\}$.

In the next subsection we will identify a large class of stochastic processes for which linear functions on $S(X)_{0,T}$ are dense in $L^p(\mathcal G)$. We view the next two results as corollaries of that result.

\begin{corollary}[Learning $F(X)$ by linear regression]
Let $Y = F(X) \in L^2(\mathcal G)$
and $\{(X^i, Y^i = F(X^i) + \varepsilon^i\}_{i = 1}^M$ be an i.i.d.\ sample of input-output pairs, where $\varepsilon^i$ are i.i.d.\ errors independent of $X$. Let $\Phi \in \mathbb R^{M \times D}$ be the data matrix with
\[
\Phi_{iv} = \langle p_v, S_N(X)_{0,T} \rangle, \qquad i = 1,\ldots, M; \ v \in \mathcal V^n
\]
with $n \leq N$ and $D = D(d, N) = \frac{d^{N+1} - 1}{d - 1}$ the dimension of $T_N(\mathbb R^d)$. Then the estimator for ordinary least squares (OLS)
\begin{equation}\label{eq:OLS}
\widehat \beta = (\Phi^\intercal \Phi )^{-1} \Phi^\intercal Y = \Phi^\intercal (\Phi\Phi^\intercal)^{-1} Y
\end{equation}
is such that $\langle \widehat \beta, S_N(X)_{0,T} \rangle$ converges in $L^2$ to $\Pi_N Y$
as $M \to \infty$. Its entries converge (in blocks) to $\Pi^n Y$, and are thus asymptotically stable under increasing degree $N$. Therefore, under the hypotheses of \Cref{thm:density} below, $Y$ can be estimated arbitrarily well by OLS linear regression on the truncated signature.
\end{corollary}

We refer to the standard literature on statistical learning (e.g.\ \cite[Ch.\ 3]{bach}) for more general statements that imply its proof. In most real-world cases in which data is limited it is desirable or even necessary to use ridge or lasso regression; consistency results could be formulated for these too, by letting the regularisation parameter vanish in the large data limit. Consistency of the OLS
estimator and the statement about universality of OLS estimation of course holds true without any orthogonalisation, since the projection $\Pi_N$ does not depend on it. However, as explained in \Cref{sec:intro}, orthogonalisation is necessary to obtain a series representation. The following is classical.

\begin{corollary}[Series expansion of an $L^2$ function on paths]\label{cor:series}
Under the hypotheses of \Cref{thm:density} below, for $Y \in L^2(\mathcal G)$
\[
Y = \sum_{n = 0}^\infty \Pi^n Y = \sum_{v: (\widetilde p_v, \widetilde p_v) > 0}\langle \ell^Y_v, S(X)_{0,T}\rangle, \qquad \ell^Y_v \coloneqq \frac{\mathbb E[Y \langle \widetilde p_v, S(X)_{0,T}\rangle]}{(\widetilde p_v, \widetilde p_v)} \widetilde p_v
\]
with $\Pi^n$ defined in \eqref{eq:PiN}, $\widetilde p_v$ in \eqref{eq:ptilde}, and convergence of the series in $L^2$.
\end{corollary}
\begin{proof}
$\Pi_N = \sum_{n = 0}^N \Pi^n$ and by density $\lVert Y - \Pi_NY \rVert_{L^2} \to 0$. The second equality follows by further decomposing $\Pi^n$ into the projections onto the orthogonal directions spanned by the vectors $\widetilde p_I$.
\end{proof}

\subsection{Density of linear functions on the signature in \texorpdfstring{$L^p$}{Lᵖ}}

The goal of this section is to prove the following result.
We equip $\R^d$ with the $\ell^1$ norm $\|(x_1,\ldots,x_d)\|=\sum_{i=1}^d|x_i|$
and equip all tensor products, including $(\R^d)^{\otimes k}$, with the projective tensor norm.

\begin{theorem}\label{thm:density}
Consider a probability measure $\mu$ on $\tG(\R^d)$
with infinite radius of convergence, i.e.
\begin{equation}\label{eq:inf_radius}
\sum_{k=0}^\infty \lambda^k \| \E_{S\sim \mu} S^k \|_{(\R^d)^{\otimes k}} < \infty \quad \text{for all } \lambda>0\;.
\end{equation}
Then the space of linear coordinate functions on $\tG(\R^d)$ are dense in $L^p(\mu)$ for all $p\in[1,\infty)$.
\end{theorem}

The exact choice of norms on $(\R^d)^{\otimes k}$ for $k\geq 1$ is not important in \Cref{thm:density} because if $\mu$ verifies \eqref{eq:inf_radius} for one choice of norms, then it verifies \eqref{eq:inf_radius} for all choices of cross norms.

As we discuss in \Cref{sec:intro}, a related $L^p$-density result, which is more restrictive as it requires time to be included as a coordinate and a weight function to be integrable, has very recently been obtained in \cite{lpdensity}. See also \Cref{thm:density_brownian} below for a short Wiener chaos-based proof that applies in the case $p = 2$ and $X$ is a time-augmented Brownian motion. \Cref{thm:density} applies to the stochastic processes to which the moment-determinacy results of \cite{characteristic} applies to, such as Gaussian and Markovian rough paths. Note that, while the hypothesis of infinite radius of convergence is the same as in \cite{characteristic}, $L^p$ density does not follow from moment determinacy alone: this is even true in finite dimension \cite{BT91} (only in the scalar case and $p \leq 2$ does this implication hold \cite[p.69]{dunkl_orthogonal_2014}). We also remark that a similar approximation result for the modified robust signatures \cite{sigMoments} was obtained in \cite{primalDual}.

For the proof, we require a few lemmas.
In the following, let $(Y,\rho)$ be a complete (not necessarily separable) metric space with a Borel probability measure $\mu$.

\begin{lemma}\label{lem:bc_dense}
The space of continuous bounded functions is dense in $L^p(\mu)$ for any $p\in [1,\infty)$.
\end{lemma}

\begin{proof}
The span of indicator functions on open sets is dense in $L^p(\mu)$.
But for $A\subset Y$ open, the continuous bounded function $f_n(y) = \{n\inf_{x\in Y\setminus A}\rho(x,y)\}\wedge 1$ converges to $\bone_A$ pointwise, and thus in $L^p$ by dominated convergence.
\end{proof}

We say that a set of functions $\cA \subset \R^{Y}$ \emph{separates points} if for all distinct $x, y \in Y$ there exists $h\in \cA$ such that $h(x)\neq h(y)$.

\begin{lemma}\label{lem:F_dense}
Suppose $\cA$ is an algebra of continuous bounded functions on $Y$ that separates points and contains the constant functions.
Suppose also that $\mu$ is Radon (which is automatic if $Y$ is separable).
Then $\cA$ is dense in $L^p(\mu)$ for any $p\in [1,\infty)$.
\end{lemma}

\begin{proof}
There are several proofs, see e.g.\ \cite{96025} for an even stronger statement that drops the continuity assumption.
We give here an elementary proof that uses only the classical Stone--Weierstrass theorem.
We let $C>0$ denote a sufficiently large universal constant.

Let $\overline\cA$ be the closure of $\cA$ in $L^p(\mu)$.
By \Cref{lem:bc_dense}, it suffices to show that if $f\colon Y\to \R$ is continuous with $\sup_{y\in Y}|f(y)| \leq 1$, then for every $\eps>0$, there exists $h \in \overline\cA$ such that $|f-h|_{L^p(\mu)} < \eps$.
To this end, consider $\eps\in(0,1)$.
Let $K\subset Y$ be a compact set such that $\mu(Y\setminus K) < \eps$;
such a set $K$ exists because $\mu$ is Radon.
By Stone--Weierstrass, there exists $g\in\cA$ such that $\sup_{y\in K} |g(y) - f(y)| < \eps$ (of course $g$ depends on $K$ and $f$).
At this stage, we \emph{don't} have $|g-f|_{L^p(\mu)}$ small because $g$ might be large outside $K$.

Now consider the functions $h_\delta(y) = e^{-\delta g(y)^2} g(y)$ for small $\delta>0$.
Define
\begin{equation*}
H(\delta) = \sup_{x\in \R}|e^{-\delta x^2}x|
\;.
\end{equation*}
Then $H(\delta)\leq C \delta^{-1/2}$
and thus
\[
\sup_{y\in Y}|h_\delta(y)| \leq H(\delta) \leq C\delta^{-1/2}\;.
\]
Moreover $\sup_{y\in K} |g(y)-h_\delta(y)| \leq C \delta$ since $\sup_{y\in K}|g(y)|\leq 2$ and $e^{-\delta z^2} = 1 + O(\delta z^2)$ for $|z|\leq 2$.
In particular, $\sup_{y\in K} |f(y)-h_\delta(y)| \leq \eps + C \delta$
and thus
\begin{equation*}
\int_{K} |f-h_\delta|^p \mrd \mu \leq (\eps +C\delta)^p\;.
\end{equation*}
On the other hand,
\begin{equation*}
\int_{Y\setminus K} \{f^p + h_\delta^p\}\mrd \mu \leq \eps + \eps H(\delta)^p \leq \eps + C^p \varepsilon \delta^{-p/2}
\end{equation*}
Taking $\delta = \eps^{1/p}$,
we obtain
\[
\|f-h_\delta\|_{L^p(\mu)}^p \leq (\eps+C\eps^{1/p})^p + \eps + C^p
\eps^{1/2}
\;.
\]
It remains to show that $h_\delta \in \overline \cA$.
We write as a power series $h_\delta = \sum_{n=0}^\infty \frac{(-\delta)^n g^{2n+1}}{n!}$ and remark that, since $g$ is bounded on $Y$, the series converges absolutely in $L^p$. Since $g^n\in\cA$ for all $n\geq 1$, we obtain $h_\delta \in \overline \cA$.
\end{proof}

\begin{proof}[Proof of \Cref{thm:density}]
Let $A$
denote the space of matrix coefficients of unitary representations of $\tG(\R^d)$ as in \cite{characteristic},
i.e. $A$ contains those functions $f\colon\tG(\R^d)\to\C$ such that $f(S) = \langle M(S) u , v\rangle_{\C^N}$ for some $u,v \in\C^N$ and a linear map $M\colon \R^d\to \mfu(N)$ for some $N\geq 1$, where $\mfu(N)$ is the space of $N\times N$ skew-Hermitian matrices and where $M(S) = \sum_{k=0}^\infty M^{\otimes k}S^k$, which is a unitary $N\times N$ matrix.
Then $A$ is a $\C$-algebra of bounded continuous functions on $\widetilde G(\R^d)$ which, by \cite[Theorem 4.8]{characteristic} (see also \cite{Hao_24} for a different proof), separates points.
Moreover $A$ contains the constant functions and is closed under complex conjugation, so by \Cref{lem:F_dense}, functions of the form $\Re f$ for $f\in A$ are dense in $L^p(\mu)$.
It remains to show that every $f\in A$ can be approximated in $L^p(\mu)$ by ($\C$-valued) linear functions on $\tG(\R^d)$.

Consider a map $M\colon \R^d \to \mfu(N)$
and let $S = (1,S^1,S^2,\ldots)$ be a $\tG(\R^d)$-valued random variable with law $\mu$.
Since $\sum_{k=0}^n M^{\otimes k}S^k$ is a linear function on $\tG(\R^d)$
and $M(S) = \sum_{k=0}^\infty M^{\otimes k}S^k$, it suffices to show that
\begin{equation}\label{eq:tail_vanish}
\lim_{n\to\infty} \sum_{k=n}^\infty \|M\|^k (\E\|S^k\|_{(\R^d)^{\otimes k}}^p)^{1/p}
= 0\;.
\end{equation}
For normed spaces $E,F$, we equip $E\oplus F$ with the $\ell^1$ norm $\|(x,y)\| = \|x\|_E + \|y\|_F$.
Consider now $n\geq 2$ and the $n$-th iterated coproduct $\Delta_n\colon T(\R^d) \to T(\R^d)^{\otimes n}$,
which is the unique algebra morphism given by $\Delta_n (v) = v\otimes 1\otimes\cdots\otimes 1 + \cdots + 1\otimes \cdots 1\otimes v$ for $v\in \R^d$, where the number of terms is $n$
(see \cite[Sec.~1.4]{reutenauer_free_1993}).
Note that $\Delta_2=\Delta$ is the usual coproduct on $T(\R^d)$ which is dual to the shuffle product.
In particular, for $v\in\R^d$,
\begin{equation}\label{eq:Delta_n_vectors}
\|\Delta_n v\|_{(\R^{d})^n} = n\|v\|_{\R^d}
\end{equation}
and thus, for $x \in (\R^d)^{\otimes k}$,
\begin{equation}\label{eq:Delta_n_norm}
\|\Delta_n x\| \leq n^k \|x\|_{(\R^d)^{\otimes k}}\;.
\end{equation}
Here, $\Delta_n x$ is an element in $\bigoplus_{k_1+\ldots+k_n=k} \bigotimes_{i=1}^n(\R^d)^{\otimes k_i}$,
which we recall is equipped with the $\ell^1$ norm of the corresponding projective tensor norms.
The bound \eqref{eq:Delta_n_norm} 
follows from \eqref{eq:Delta_n_vectors}
and the fact that, for any normed algebra $B$ and linear map $Q\colon\R^d\to B$ with operator norm $\|Q\|$, $Q^{\otimes k}\colon (\R^d)^{\otimes k}\to B$ has operator norm bounded above by $\|Q\|^k$ due to the choice of projective norms on $(\R^d)^{\otimes k}$.
More precisely, we apply this with $Q=\Delta_n$, for which $\|Q\|=n$ by \eqref{eq:Delta_n_vectors}, and $B= T(\R^d)^{\otimes n}$ equipped with the projective norm, where we equip $T(\R^d)$ with the $\ell^1$ norm $\|(x^0,x^1,\ldots)\| = \sum_{k\geq0}\|x^k\|_{(\R^d)^{\otimes k}}$.

Therefore, for $p\geq 1$ an integer,
using the duality between shuffle and $\Delta$ as in \cite[Sec.~1.5]{reutenauer_free_1993} and the fact that $S$ takes values in $G(\R^d)$,
\begin{align*}
\E \|S^k\|_{(\R^d)^{\otimes k}}^{2p}
&=
\E \Big(\sum_{|w|=k}\lvert\scal{w,S^k}\rvert\Big)^{2p}
\leq
d^{k(2p-1)}  \sum_{|w|=k} \E \scal{w,S^k}^{2p}
=
d^{k(2p-1)}  \sum_{|w|=k} \E \scal{w^{\shuffle 2p},S^{2pk}}
\\
&=
d^{k(2p-1)} \sum_{|w|=k} \E \scal{w^{\otimes 2p}, \Delta_{2p} S^{2pk}}
\leq
d^{k(2p-1)} \|\Delta_{2p} \E S^{2pk}\|
\\
&\leq
d^{k(2p-1)} (2p)^{2pk} \|\E S^{2pk}\|_{(\R^d)^{\otimes 2pk}}\;,
\end{align*}
where $w$ in the sums ranges over all words of length $k$ in the canonical basis of $\R^d$
(this generalises the calculation above \cite[Proposition 3.4]{characteristic} which was for $p=1$).
Since $\E S$ has an infinite radius of convergence, for any $\eps>0$, one has
$\|\E S^{k}\| \leq \eps^k$ for all $k>k(\eps)\geq 1$ sufficiently large.
Therefore, for all $k>k(\eps)$,
\begin{equation*}
(\E \|S^k\|^{2p})^{1/(2p)} \leq d^\frac{k(2p-1)}{2p}(2p)^k \eps^{k}\;.
\end{equation*}
Taking $\eps$ sufficiently small (depending on $\|M\|$, $p$, and $d$),
we obtain \eqref{eq:tail_vanish} as required.
\end{proof}

\section{General properties of orthogonal shuffle polynomials}\label{sec:orthPol}

In this section, we discuss structural properties of orthogonal shuffle polynomials in $\Sh(V^*)$, generalizing several classical results in the classical commutative setting. We begin by reformulating shuffle polynomials as ordinary commutative polynomials with graded generators. 

\subsection{Shuffle polynomials as graded commutative polynomials}

For a vector space $V$, let $\Sym(V)$ denote the symmetric algebra of $V$.
Recall that $L(V)$ denotes the free Lie algebra of $V$ and, from \cite{reutenauer_free_1993},
\begin{equation}
\mathrm{Sh}(V^*) = \Sym(L(V)^\gr)
\end{equation}
as commutative algebras, where $(\cdot)^\gr$ denotes the graded dual and identity is intended as being between functors. Fixing a Hall basis of $L(V)$ allows us to express elements of $\Sh(V^*)$ as polynomials over an infinite set of indeterminates, specified by the dual Hall basis \cite[\S 5.2]{reutenauer_free_1993}. As many of the results in this section do not rely on the Lie algebra structure, we will mainly work with a general graded vector space $W = \bigoplus_{m=0}^\infty W^m$ of finite type, where the degree $m$ subspace $W^m$ is finite dimensional. Our general results for $\Sym(W^\gr)$ apply to the signature setting by taking $W = L(V)$.\medskip



We fix a basis of $W^m$, denoted $\{w_{m,i}\}_{i=1}^{N_m}$, where $N_m$ is the dimension of $W^m$. We abuse notation, and also denote the dual basis in $(W^m)^*$ by $w_{m,i}$. In general, we will denote a monomial in $\Sym(\gvec^\gr)$
by $\sbasvec^\balpha$, where $\balpha = (\alpha_{m, i})_{m \in \N, \, i \in [N_m]}$, where each $\alpha_{m,i} \in \N$, and only finitely many are nonzero, and set 
\begin{align} \label{eq:l_monomials}
    \sbasvec^\balpha = \prod_{m =1}^\infty \,\prod_{i =1}^{N_m} w^{\alpha_{m,i}}_{m,i}. 
\end{align}
Let $Q \in \Sym(\gvec^\gr)$
be a polynomial $Q = \sum_{\balpha} q_\balpha \sbas^\balpha$, where finitely many of the $b_\alpha$ are nonzero. 
We have two notions of degree called the \emph{tensor (total) degree} and the \emph{shuffle degree}, respectively defined on a shuffle monomial $\sbasvec^\balpha$ by
\[
    \tdeg(\sbas^\balpha) = |\balpha| = \sum_{m=1}^\infty \, \sum_{i=1}^{N_m} m \alpha_{m,i} \andd \sdeg(\sbas^\balpha) =  \sum_{m=1}^\infty \, \sum_{i=1}^{N_m} \alpha_{m,i}.
\]
We denote the column vector of shuffle monomials of total degree $n$ by $\sbasvec_n$ and denote its dimension by $r_n$.


\begin{definition}
    A linear functional $\cL : \Sym(W^\gr) \to \R$ is \emph{quasi-definite} if there exists a basis of $\Sym(W^\gr) $ such that for any two basis elements $P, Q$, we have
    \[
        \cL(PQ) = 0 \text{ if } P \neq Q \quad \andd \quad \cL(P^2) \neq 0.
    \]
    If in addition, $\cL$ satisfies $\cL(P^2) > 0$, then we say that $\cL$ is positive-definite.
\end{definition}

Given a quasi-definite functional $\cL$, define the symmetric bilinear form (or inner product, if $\cL$ is positive-definite)
\[
    (P,Q) \coloneqq \cL(PQ),
\]
and say that $\{\bp_n\}_{n\geq 0}$, where $\bp_n \in \Sym(W^{\gr})^{r_n}$ is \emph{block orthogonal} if $(\bw_m, \bp_n^\tr) = 0$ for all $m < n$ and $(\bw_n, \bp_n^\tr)$ is invertible. In this section, orthogonal polynomials refer to block orthogonal polynomials. 
We say that $\{\tbp_n\}_{n\geq 0}$ is \emph{orthonormal} if in addition, $(\tbp_n, \tbp_n^\tr)$ is the identity. 
We note that there are two main differences from the classical orthogonal polynomial setting:
\begin{enumerate}
    \item there are an infinite number of generators, ie.~$W$ is infinite-dimensional; and
    \item the generators are \emph{graded}, and thus orthogonalize with respect to the \emph{total} degree.
\end{enumerate}

In the following subsections, we generalize classical results about orthogonal polynomials~\cite{dunkl_orthogonal_2014} to our graded, infinite-dimensional setting. We show that orthogonal polynomials on $\Sym(W^\gr)$ satisfy a recurrence relation along with rank conditions on the defining matrices. In~\Cref{ssec:favard}, we prove the converse in a generalization of Favard's theorem, and furthermore discuss when inner products on $\Sym(W^\gr)$ are induced by probability measures in~\Cref{ssec:jacobi}. We return to the specific case of $W = L(V)$
in~\Cref{ssec:measure_GV} and discuss measures on $G(V)$.

\subsection{Recurrence relation} \label{ssec:recurrence}
We begin by assuming that $\cL$ is a quasi-definite linear functional 
on $\Sym(W^\gr)$ and that $\{\bp_n\}_{n\geq 0}$ denotes a system of block orthogonal polynomials.
Let
\[
    H_n = (\bp_n, \bp_n^\tr) \in \Mat(r_n, r_n)
\]
which is symmetric and invertible by definition of block orthogonality.
We begin by generalizing the three-term relation of~\cite[Theorem 3.3.1]{dunkl_orthogonal_2014} to the graded setting.


\begin{proposition} \label{prop:relation}
    For $n \in \N_0$, $m \in [n]$, $i \in [N_m]$ and $-m \leq k \leq m$, there exist unique matrices $M^k_{n, m, i} \in \Mat(r_{n-m}, r_{n-m+k})$ such that
    \begin{align} \label{eq:recurrence}
        \sbas_{m,i} \bp_{n-m} = \sum_{k=-m}^{m} M^k_{n,m,i} \bp_{n-m+k}.
    \end{align}
    These matrices satisfy
    \begin{align} \label{eq:M_identity}
        M^k_{n,m,i} H_{n-m+k} = H_{n-m} (M^{-k}_{n+k,m,i})^\tr.
    \end{align}    
\end{proposition}

\begin{proof}
    The components of $\sbas_{m,i} \bp_{n-m}$ are polynomials of degree $n$, so they can be written as a linear combination of orthogonal polynomials,
    \[
        \sbas_{m,i} \bp_{n-m} = \sum_{k=0}^{n} M^k_{n,m,i} \bp_k.
    \]
    Next, due to the orthogonality of polynomials, we have
    \begin{align} \label{eq:orthogonality_recurrence}
        (\sbas_{m,i} \bp_{n-m}, \bp_k^\tr) = \begin{cases} M^k_{n,m,i} (\bp_k, \bp_k^\tr) & :k =n-2m, \ldots, n \\
            0 & : k < n - 2m.
        \end{cases}
    \end{align}
    The second case is trivial since $\sbas_{m,i} \bp_k^\tr$ is at most degree $k+m \leq n-m$ and
    \[
        (\sbas_{m,i} \bp_{n-m}, \bp_k^\tr) = (\bp_{n-m}, \sbas_{m,i} \bp_k^\tr) = 0
    \]
    Therefore, we have
    \begin{align} \label{eq:M_def}
        M^k_{n,m,i} = \begin{cases} (\sbas_{m,i} \bp_{n-m}, \bp_k^\tr) H_k^{-1} & : k =n-2m, \ldots, n \\
            0 & : k < n - 2m.
        \end{cases}
    \end{align}
    Furthermore, for $k = n-2m, \ldots, n$, we have
    \[
        M^k_{n,m,i} H_k = (\sbas_{m,i} \bp_{n-m}, \bp_k^\tr) = (\sbas_{m,i} \bp_k, \bp_{n-m}^\tr)^\tr =  H_{n-m} (M^{n-m}_{k+m,m,i})^\tr.
    \]
    Then we reindex by $k \mapsto n-m+k$ to get the result.
\end{proof}

Rather than the classical three term recurrence, we now obtain a recurrence relation with $2m+1$ terms for degree $m$ generators. This is due to the fact that multiplication by $w_{m,i}$ increases the degree by $m$, and thus $(w_{m,i} \bp_{n-m}, 
bp_k^\tr)$ may be non-trivial for a larger range of degrees $k$, see~\eqref{eq:orthogonality_recurrence}.
Here,~\eqref{eq:M_identity} provides relationships between matrices for different orders of $n$.
Two of these will be of particular importance, and we will rename them as
\begin{align}
    A_{n,m,i} \coloneqq M^{m}_{n,m,i} \in \Mat(r_{n-m}, r_{n}) \quad \andd \quad C_{n,m,i} \coloneqq M^{-m}_{n+m,m,i} \in \Mat(r_{n}, r_{n-m}),
\end{align}
which are related through~\eqref{eq:M_identity} by
\begin{align} \label{eq:A_C_relation}
    A_{n,m,i} H_{n} = H_{n-m} C_{n,m,i}^\tr.
\end{align}
In particular, $A_{n,m,i}$ is the leading matrix in the recurrence relation for $\sbas_{m,i} \bp_{n-m}$, while $C_{n,m,i}$ is the non-leading matrix in the recurrence for $\sbas_{m,i} \bp_{n}$ which relates to $A_{n,m,i}$ via~\eqref{eq:A_C_relation}.
So far, we have been working with a block orthogonal system of polynomials $\bp_n$. However, if the system of polynomials is \emph{orthonormal}, which we denote by $\tbp_n$, then $H_n=I$ is the identity matrix, and we obtain the following.


\begin{corollary} \label{cor:relation_on}
    Let $\tbp_n$ be an orthonormal system of polynomials. For $n \in \N_0$, $m \in [n]$, $i \in [N_m]$ and $0 \leq k \leq m$, there exist unique matrices $M^k_{n, m, i} \in \Mat(r_{n-m}, r_{n-m+k})$ for such that
    \begin{align} \label{eq:recurrence_on}
        \sbas_{m,i} \bp_{n-m} = \sum_{k=0}^{m} M^k_{n,m,i} \tbp_{n-m+k} + \sum_{k=1}^{m} (M^{   k}_{n-k,m,i})^\tr \tbp_{n-m-k}.
    \end{align}
    Furthermore, $M^{n-m}_{n,m,i}$ is symmetric. 
\end{corollary}
\begin{proof}
    This is immediate from~\Cref{prop:relation} since $H_n = I$ is the identity for all $n$.
\end{proof}

\subsection{Rank condition}
Now, we will consider rank conditions that the matrices $A_{n,m,i}$ and $C_{n,m,i}$ must satisfy.
We will consider \emph{column joint matrices}, which stack a set of matrices in a column. For fixed $n \in \N$, suppose $\{Q_{n,m,i}\}_{m \in [n], \, i \in [N_m]}$ is a set of matrices where $Q_{n,m,i}$ is of size $(p_m, q)$,
then we define
\begin{align} \label{eq:joint_matrix}
    Q_{n,m} \coloneqq \pmat{Q_{n,m,1} \\ Q_{n,m,2} \\ \vdots \\ Q_{n,m,N_m}} \quad \andd \quad Q_n \coloneqq \pmat{Q_{n,1} \\ Q_{n,2} \\ \vdots \\ Q_{n, n}}.
\end{align}
Here $Q_{n,m} \in \Mat(N_m p_m, q)$ and $Q_n \in \Mat(K, q)$, where $K = \sum_{m=1}^{n} N_m p_m$. One particular joint matrix that will be important is $A_n \in \Mat(R_{n}, r_{n})$, built from $A_{n,m,i} \in \Mat(r_{n-m}, r_{n})$, where
\begin{align}
    R_{n} = \sum_{k=1}^{n} N_k \cdot r_{n-k}.
\end{align} 

Let $G_n \in \Mat(r_n, r_n)$ be the leading-coefficient matrix of $\bp_n$; that is, the matrix $G_n$ such that\footnote{Note that these are leading coefficients in terms of the shuffle monomials (with respect to total degree). Thus, even for the monic polynomials defined by Gram-Schmidt in~\eqref{eq:ptilde}, $G$ is not necessarily the identity.} 
\begin{align}
    \bp_n = G_n \bw_n + \bq,
\end{align}
where $\bw_n$ consist of total degree $n$ monomials, and $\bq$ has total degree strictly less than $n$.
Next, let $L_{n,m,i} \in \Mat(r_{n-m}, r_{n})$ defined by
\[
    L_{n,m,i} \sbasvec_{n} = \sbas_{m,i} \cdot \sbasvec_{n-m}.
\]

\begin{lemma} \label{lem:L_rank}
    For each $n \in \N_0$, $m \in [n]$ and $i \in [N_m]$, the matrix $L_{n,m,i}$ satisfies
    \[
        L_{n,m,i} \cdot L_{n,m,i}^\tr = I.
    \]
    Moreover,
    \[
        \rank(L_{n,m,i}) = r_{n-m} \quad \andd \quad \rank(L_n) = r_{n}.
    \]
\end{lemma}
\begin{proof}
    By definition, each row of $L_{n,m,i}$ has exactly one element equal to $1$; otherwise the elements are $0$. Thus, we have $L_{n,m,i} \cdot L_{n,m,i}^\tr = I$ and $\rank(L_{n,m,i}) = r_{n-m}$. Now, let
    \[
        \cN_n \coloneqq \{ \balpha \in \N_0^{N_{\leq n}} \, : \, |\balpha| = n\} \quad \andd \quad \cN_{n,m,i} \coloneqq \{ \balpha \in \cN_n \, : \, \alpha_{m,i} \neq 0\},
    \]
    where the $\balpha$ denotes a graded exponent vector. Then, let $\ba = (a_\balpha)_{\balpha \in \cN_{n}} \in \Mat(r_{n}, 1)$ be a column vector. Then, $L_{n,m,i}$ is a map which projects $\ba$ onto its restriction onto $\cN_{n, m,i}$. Now, if $L_n \ba = 0$, then $L_{n,m,i} \ba = 0$ for all $m \in [n]$ and $i \in [N_m]$. However, since $\cup_{m,i} \cN_{n,m,i} = \cN_n$, we have $\ba = 0$, so $L_n$ has full rank. 
\end{proof}

Next, for fixed $n \in \N$, we compare the leading coefficient matrices on each side of~\eqref{eq:recurrence} to get
\begin{align} \label{eq:M_L_relation}
    G_{n-m} L_{n,m,i} = A_{n,m,i} G_{n}
\end{align}
for any $m \in [n]$ and $i \in [N_m]$. 

\begin{proposition} \label{prop:rank_condition}
    For each $n \in \N$, $m \in [n]$ and $i \in [N_m]$,
    \begin{align} \label{eq:A_C_rank_individual}
        \rank(A_{n,m,i}) = \rank(C_{n,m,i}) = r_{n-m}.
    \end{align}
    Let $A_{n} \in \Mat(R_{n}, r_{n})$ be the joint matrix of $A_{n,m,i}$ and $C_n^\tr \in \Mat(R_n, r_n)$ be the joint matrix of $C_{n,m,i}^\tr \in \Mat(r_{n-m}, r_n)$. Then, 
    \begin{align} \label{eq:A_C_rank_joint}
        \rank(A_n) = \rank(C_n^\tr) = r_{n}.
    \end{align}
\end{proposition}
\begin{proof}
    From~\eqref{eq:M_L_relation} and since all $G_n$ are invertible, $\rank(A_{n,m,i}) = r_{n-m}$ from~\Cref{lem:L_rank}. Furthermore, since $H_n$ is invertible, we also have $\rank(C_{n,m,i}) = r_{n-m}$ from~\eqref{eq:M_identity}. Now, we define 
    \begin{align} \label{eq:bGm}
        \bG_{m} \coloneqq G_m^{\oplus N_{n-m}} \in \Mat(N_{n-m} r_{m}, N_{n-m} r_{m}) \andd \bG \coloneqq \bG_{n-1} \oplus \ldots \oplus \bG_{0} \in \Mat(R_{n}, R_{n}).
    \end{align}
    Let $L_n \in \Mat(R_{n}, r_{n})$ be the joint matrix of $L_{n,m,i}$.
    Then, from~\eqref{eq:M_L_relation} and the definition of the joint matrix $A_n$ from~\eqref{eq:joint_matrix}, we have $\bG L_n = A_n G_{n}$, written out in matrix form as 
    \begin{align} \label{eq:leading_coefficient_relation}
        \bG L_n = \pmat{\bG_{n-1} & 0 & \ldots & 0 \\ 0 & \bG_{n-2} & \ldots & 0 \\ \vdots & \vdots &\ddots & \vdots \\ 0 & 0 &\ldots & \bG_0} \pmat{L_{n,1} \\ L_{n,2} \\ \vdots \\ L_{n,n}} = \pmat{A_{n,1} \\ A_{n,2} \\ \vdots \\ A_{n,n}} G_{n} = A_n G_n
    \end{align}
    
    Note that $\bG$ is invertible and thus $\rank(A_n) = \rank(L_n) = r_{n}$.
    Next, we define
    \begin{align} \label{eq:bHm}
        \bH_{m} \coloneqq H_m^{\oplus N_{n-m}} \in \Mat(N_{n-m} r_{m}, N_{n-m} r_{m}) \andd \bH \coloneqq \bH_{n-1} \oplus \ldots \oplus \bH_{0} \in \Mat(R_{n}, R_{n}). 
    \end{align}
    Then,~\eqref{eq:A_C_relation} implies $A_n H_n = \bH C_n^\tr$. Because $\bH$ is invertible, we get $\rank(C_n^\tr) = \rank(A_n) = r_n$.
\end{proof}

Because the matrix $A_n$ has full rank, there exists a generalized left inverse $D_n^\tr \in \Mat(r_n, R_n)$, which is expressed as the row joint matrix of $D_{n,m,i}^\tr \in \Mat(r_n, r_{n-m})$. More explicitly,
\[
    D_{n,m}^\tr \coloneqq \pmat{D_{n,m,1}^\tr & \ldots & D_{n,m,N_m}^\tr} \quad \andd \quad D_n^\tr \coloneqq \pmat{D_{n,1}^\tr & \ldots & D_{n,n}^\tr}.
\]
This generalized left inverse (which may not be unique) satisfies
\begin{align} \label{eq:D_left_inverse}
    D_n^\tr A_n = \sum_{m=1}^n \sum_{i=1}^{N_m} D_{n,m,i}^\tr \cdot A_{n,m,i} = I \in \Mat(r_n, r_n).
\end{align}

\begin{proposition} \label{prop:gen_inverse_relation}
    Let $D^\tr_n$ be a generalized inverse of $A_n$. Then, 
    \[
        \bp_{n} = \sum_{m=1}^{n} \sum_{i=1}^{N_m} \left(\sbas_{m,i} D_{n,m,i}^\tr \bp_{n-m}- \sum_{k=-m}^{m-1} D^\tr_{n,m,i} M^k_{n,m,i} \bp_{n-m+k}\right)
    \]
\end{proposition}
\begin{proof}
    We multiply the relation from~\eqref{eq:recurrence} by $D_{n,m,i}^\tr$ to get
    \[
        \sbas_{m,i} D^\tr_{n,m,i} \bp_{n-m} = D^\tr_{n,m,i} A_{n,m,i} \bp_n + \sum_{k=-m}^{m-1} D^\tr_{n,m,i}M^k_{n,m,i} \bp_{n-m+k}.
    \]
    Next, we sum over $m \in [n]$ and $i \in [N_m]$ to get
    \[
        \sum_{m=1}^{n} \sum_{i=1}^{N_m} \left(\sbas_{m,i} D_{n,m,i}^\tr \bp_{n-m}\right) = \left(\sum_{m=1}^{n} \sum_{i=1}^{N_m} D^\tr_{n,m,i} A_{n,m,i}\right)\bp_n +  \sum_{m=1}^{n} \sum_{k=-m}^{m-1} \left( \sum_{i=1}^{N_m} D^\tr_{n,m,i} M^k_{n,m,i} \right) \bp_{n-m+k}.
    \]
    Then, applying~\eqref{eq:D_left_inverse}, we obtain the desired result. 
\end{proof}

\subsection{Favard's theorem} \label{ssec:favard}

In the previous sections, we showed that given a quasi-definite functional on $\Sym(W^\gr)$, orthogonal polynomials satisfy a recurrence relation where the defining matrices must satisfy certain rank conditions. Here, we will prove the converse: a generalization of Favard's theorem.


\begin{theorem} \label{thm:favard_quasi}
    Let $\bp = \{\bp_n\}_{n=0}^\infty$ be an arbitrary sequence where $\bp_n \in \Sym(W^\gr)^{r_n}$ is a column vector of length $r_n$, and $\bp_0 = 1$. Then, the following statements are equivalent.
    \begin{enumerate}
        \item There exists a quasi-definite $\cL$ under which $\{\bp_n\}_{n=0}^\infty$ is a block orthogonal basis in $\Sym(W^\gr)$.
        \item For $n \in \N_0$, $m \in [n]$, $i \in [N_m]$ and $-m \leq k \leq m$, there exist  $M^k_{n, m, i} \in \Mat(r_{n-m}, r_{n-m+k})$ where
        \begin{enumerate}
            \item the polynomials $\bp_n$ satisfy the relation in~\eqref{eq:recurrence}, and
            \item the matrices $A_{n,m,i} \coloneqq M^{m}_{n,m,i}$ and $C_{n,m,i} \coloneqq (M^{-m}_{n+m,m,i})^\tr$ satisfy~\eqref{eq:A_C_rank_individual} and~\eqref{eq:A_C_rank_joint}.
        \end{enumerate}
    \end{enumerate}
\end{theorem}
\begin{proof}
    The forward direction is given by~\Cref{prop:relation,prop:rank_condition}. We will now prove the converse direction. \medskip

    \textbf{$\bp$ is a basis.} To begin, we must show that $\bp$ forms a basis of $\Sym(W^\gr)$; it suffices to show that the leading coefficient matrix $G_n$ of $\bp_n$ is invertible. Because the polynomials satisfy the relation in~\eqref{eq:recurrence}, the leading coefficient matrices satisfy $\bG L_n = A_n G_n$ from~\eqref{eq:leading_coefficient_relation} (and using the same definition of $\bG$ from there). We will now show that $G_n$ is invertible by induction. For $n=0$, we have $\bp_0 = 1$, so $G_0 = 1$. Now, suppose $G_0, \ldots, G_{n-1}$ are invertible. By the definition of $\bG$ in~\eqref{eq:bGm}, this implies that $\bG$ is invertible. Therefore, we have
    \[
        \rank(A_n G_n) = \rank(\bG L_n) = \rank(L_n) = r_n,
    \]
    by~\Cref{lem:L_rank}. Then, using the rank hypothesis of $\rank(A_n) = r_n$, and the rank inequality for product matrices, we have
    \[
        \rank(G_n) \ge \rank(A_n G_n) \ge \rank(A_n) + \rank(G_n) - r_n = \rank(G_n).
    \]
    Thus, $\rank(G_n) = \rank(A_n G_n) = r_n$, so $G_n$ is invertible. \medskip

    \textbf{$\bp$ is block orthogonal.} Because $\bp$ is a basis of $\Sym(W^\gr)$, we can define $\cL: \Sym(W^\gr) \to \R$ by
    \[
        \cL(1) = 1, \quad \cL(\bp_n) = 0 \text{ for } n \ge 1.
    \]
    Now, we will show that $\bp$ is block orthogonal with respect to this linear functional. In particular, we will use induction to show that
    \begin{align} \label{eq:favard_orthog_condition}
        \cL(\bp_k \bp_j^\tr) = 0 \quad \text{for} \quad k \neq j.
    \end{align}
    Suppose~\eqref{eq:favard_orthog_condition} holds for $k,j$ such that $0 \leq k \leq n-1$ and $j > k$. Note that this holds by definition when $n =1$. Now, we note that~\Cref{prop:gen_inverse_relation} only used the recurrence in~\eqref{eq:recurrence} and the rank condition in~\eqref{eq:A_C_rank_joint}. Thus, for $\ell > n$, we have
    \begin{align*}
        \cL(\bp_{n}\bp_\ell^\tr) &= \cL \left(\sum_{m=1}^{n} \sum_{i=1}^{N_m} \sbas_{m,i} D_{n,m,i}^\tr \bp_{n-m} \bp_\ell^\tr\right) = \cL \left(\sum_{m=1}^{n} \sum_{i=1}^{N_m}  D_{n,m,i}^\tr \bp_{n-m} \left(\sum_{k=-m}^{m} M^k_{\ell+m, m,i} \bp_{\ell+k}\right)^\tr\right) = 0,
    \end{align*}
    where we use~\Cref{prop:gen_inverse_relation} and the induction hypothesis in the first equality, the relation in~\eqref{eq:recurrence} on $\sbas_{m,i} \bp_\ell^\tr$ in the second, and the induction hypothesis in the third. Thus,~\eqref{eq:favard_orthog_condition} holds. \medskip

    \textbf{$\cL$ is nondegenerate.} Now, we will show that $H_n = \cL(\bp_n \bp_n^\tr)$ is invertible. First, we note that $H_n$ is symmetric by definition. Next, by block orthogonality,~\eqref{eq:M_def} holds, and therefore~\eqref{eq:A_C_relation} also holds. Now, using the definition of $\bH$ from~\eqref{eq:bHm}, we have
    \begin{align} \label{eq:favard_quasi_An_Cn}
        A_n H_n = \bH C_n^\tr.
    \end{align}
    We will once again show that $H_n$ is invertible by induction. Because $\cL(1) = 1$, we have $H_0 = 1$, which is invertible. Now, suppose that $H_0, \ldots, H_{n-1}$ is invertible, which implies that $\bH$ is invertible. Then, using the rank condition $\rank(C_n^\tr) = r_n$ from~\eqref{eq:A_C_rank_joint} by hypothesis, we have
    \[
        \rank(A_n H_n) = \rank(\bH C_{n}^\tr) = \rank(C_n^\tr) = r_n
    \]
    Then, using the rank inequality for product matrices and the rank condition for $A_n$ again, 
    \[
        \rank(H_n) \ge \rank(A_n H_n) \ge \rank(A_n) + \rank(H_n) - r_n = \rank(H_n).
    \]
    Thus, $\rank(H_n) = \rank(A_n H_n) = r_n$, which implies that $H_n$ is invertible. \medskip

    Finally, the fact that $H_n$ is invertible implies that $\cL$ is a quasi-definite linear functional which makes $\bp$ a block orthogonal basis of $\Sym(W^\gr)$. 
\end{proof}

Next, we will show Favard's theorem for positive-definite functionals.

\begin{theorem} \label{thm:favard_positive}
    Let $\bp = \{\tbp_n\}_{n=0}^\infty$ be an arbitrary sequence where $\tbp_n \in \Sym(W^\gr)^{r_n}$ is a column vector of length $r_n$, and $\tbp_0 = 1$. Then, the following statements are equivalent.
    \begin{enumerate}
        \item There exists a positive-definite $\cL$ under which $\{\tbp_n\}_{n=0}^\infty$ is an orthonormal basis in $\Sym(W^\gr)$.
        \item For $n \in \N_0$, $m \in [n]$, $i \in [N_m]$ and $0 \leq k \leq m$, there exist matrices $M^k_{n, m, i} \in \Mat(r_{n-m}, r_{n-m+k})$ such that
        \begin{enumerate}
            \item the polynomials $\tbp_n$ satisfy the relation in~\eqref{eq:recurrence_on}, and
            \item the matrices $A_{n,m,i} \coloneqq M^{m}_{n,m,i}$ and $C_{n,m,i} \coloneqq (M^{m}_{n-m,m,i})^\tr$ satisfy~\eqref{eq:A_C_rank_individual} and~\eqref{eq:A_C_rank_joint}.
        \end{enumerate}
    \end{enumerate}
\end{theorem}
\begin{proof}
The forward direction is given by~\Cref{cor:relation_on} and~\Cref{prop:rank_condition}. We will now prove the converse direction. 
    From~\Cref{thm:favard_quasi}, there exists a quasi-definite linear functional $\cL$ such that $\tbp_n$ is a block orthogonal basis in $\Sym(W^\gr)$, so we only need to show that $\cL$ is positive-definite. Let $H_n = \cL(\tbp_n \tbp_n^\tr)$; it suffices to show that this is the identity for every $n \in \N_0$. Because $\tbp_0 = 1$ and $\cL(1) = 1$, we have $H_0 = 1$, and we proceed by induction on $n$. Suppose $H_k = I$ is the identity for $k \leq n-1$. Following the previous proof, $C_n = (A_n)^\tr$ implies that~\eqref{eq:favard_quasi_An_Cn} becomes
    \[
        A_n H_n = \bH A_n.
    \]
    By definition of $\bH$ from~\eqref{eq:bHm} and induction, we have $\bH = I$, so $H_n = I$ is the identity.
\end{proof}

\subsection{Jacobi matrices and measures on \texorpdfstring{$\hW$}{Ŵ}} \label{ssec:jacobi}
Favard's theorem in~\Cref{thm:favard_positive} only provides the existence of a positive-definite linear functional on $\Sym(W^\gr)$.
In this section, we consider the question of when the functional $\cL$ is induced by a measure $\mu \in \cM(\hW)$ on $\hW = \prod_{m=0}^\infty W^m$, 
\begin{align} \label{eq:measure_functional}
    \cL(P) = \int_{\hW} P(w) d\mu(w).
\end{align}
In particular, for the remainder of this section, we will work exclusively with positive-definite linear functionals, and the corresponding recurrence matrix conditions in~\Cref{thm:favard_positive}.
We do this by generalizing the notion of \emph{Jacobi matrices}, and use the spectral theory of commuting self adjoint operators (CSOs). In this section, we will be exclusively working with \emph{orthonormal polynomials}, and thus use the recurrence relation in~\Cref{cor:relation_on}. \medskip

Suppose we have a positive-definite linear functional $\cL$ on $\Sym(W^\gr)$, and let $\tbp = (\tbp_n)_{n=0}^\infty$ denote a system of orthonormal polynomials with respect to this functional. Furthermore, let $H(W^\gr) \cong \ell^2$ denote the Hilbert space completion of $\Sym(W^\gr)$, equipped with an orthonormal basis\footnote{We treat $H(W^\gr)$ as an abstract Hilbert space, and use the notation $\phi_\balpha$ to forget that these elements are polynomials.} denoted $\phi_\balpha$ as in~\eqref{eq:l_monomials}. We let $\Phi_n = (\phi_\balpha)_{|\balpha|=n}$ denote the column vector of total degree $n$ elements. 
Then, for each $m \in \N$ and $i \in [N_m]$, we collect all recurrence matrices $M_{n,m,i}^k$ into the  \emph{Jacobi matrix} $J_{m,i}$, which is a linear operator in sequence space $H(W^\gr)$.
This is defined as a semi-infinite band-diagonal block matrix, where the width of the band is $2m+1$, as
\begin{align} \label{eq:jacobi_matrix}
    J_{m,i} \coloneqq \pmat{ 
        M^0_{m,m,i} & M^1_{m,m,i} & M^2_{m,m,i} & \cdots & M^m_{m,m,i} & 0 & 0 & \cdots \\
        (M^1_{m,m,i})^\tr & M^0_{m+1,m,i} & M^1_{m+1,m,i} & \cdots & M^{m-1}_{m+1,m,i} & M^{m}_{m+1,m,i} & 0 & \cdots \\
        (M^2_{m,m,i})^\tr & (M^1_{m+1,m,i})^\tr & M^0_{m+2,m,i} & \cdots & M^{m-2}_{m+2,m,i} & M^{m-1}_{m+2,m,i} & M^{m}_{m+2,m,i} & \cdots \\
        \vdots & \vdots & \vdots & & \vdots & \vdots & \vdots &
    },
\end{align}
where the $M^k_{n,m,i}$ are given in~\Cref{cor:relation_on}. 
Note that the diagonal blocks are symmetric matrices.
Now, our aim is to show there exists a probability measure $\mu$ on $\hW$ such that $J_{m,i}$ is unitarily equivalent to the multiplication by $w_{m,i}$ operator on $L^2(\hW, \mu)$. \medskip

\subsubsection{Background on spectral theory}

We do this by using the spectral theory of a countable family of CSOs. The following summary of the relevant theory is from~\cite{samoilenko_spectral_2013}. Suppose $\cH$ is a separable Hilbert space with inner product $\langle \cdot, \cdot \rangle$. Every self-adjoint operator $T: \cH \to \cH$ has a spectral measure $E$ on $\R$. In particular, $E$ is a projection-valued Borel measure such that $T = \int_\R \lambda dE(\lambda)$, $E(\R)$ is the identity operator and $E(B \cap C) = E(B) \cap E(C)$. We say that a family $\{T_j\}_{j=1}^\infty$ of self-adjoint operators \emph{strongly commute} if their spectral measures commute,
\[
    E_i(B)E_j(C) = E_i(C) E_j(B)
\]
for any Borel sets $B,C \subset \R$. If a finite family $\{T_j\}_{j=1}^d$ commute, then the \emph{spectral measure of the commuting family $T_1, \ldots, T_d$} is a projection valued measure on $\R^d$ defined by
\[
    E(B_1 \times \ldots \times B_d) = E_1(B_1) \cdots E_d(B_d). 
\]
If we have a countably infinite family $\{T_j\}_{j=1}^\infty$, we want to obtain a spectral measure on the vector space of sequences $\R^\infty$ equipped with the product topology. Let $\cB(\R^\infty)$ denote its Borel $\sigma$-algebra.

\begin{definition}{\cite[Definition 1]{samoilenko_spectral_2013}}
    An operator-valued measure $E$ defined on $(\R^\infty, \cB(\R^\infty))$ is a \emph{spectral measure} (or a \emph{resolution of the identity}) if:
    \begin{enumerate}
        \item \textbf{(projection)} $E(B)$ is a projection on $\cH$ for all $B \in \cB(\R^\infty)$, $E(\emptyset) = 0$ and $E(\R^\infty) = I$;
        \item \textbf{(additivity)} if $\{B_j\}_{j=1}^\infty$ are mutually disjoint, then $E(\bigcup_{j=1}^\infty B_j) = \sum_{j=1}^\infty E(B_j)$;
        \item \textbf{(orthogonality)} for all $B, B' \in \cB(\R^\infty)$, we have $E(B \cap B') = E(B) E(B')$.
    \end{enumerate}
\end{definition}

\begin{theorem}{\cite[Theorem 1]{samoilenko_spectral_2013}} \label{thm:spectral1}
    For every countable family $\{T_j\}_{j=1}^\infty$ of CSOs, there exists a unique spectral measure $E$. Conversely, every spectral measure is generated by a countable family of CSOs where
    \begin{align}
        T_j = \int_{\R^\infty} \lambda_j dE(\lambda_1, \lambda_2, \ldots) = \int_\R \lambda_j dE_j(\lambda_j),
    \end{align}
    where $E_j$ is the 1D spectral measure $E_j(B) = E(\R \times \ldots \times \R \times B \times \R \times \ldots)$, where $B$ is in the $j^{th}$ spot. 
\end{theorem}
For a countable family $\{T_j\}_{j=1}^\infty$, a \emph{cyclic vector} is an element $\Phi_0 \in \cH$ such that
\begin{align} \label{eq:cyclicvec_general}
    \overline{\text{span}\{E(B) \Phi_0 \, : \, B \in \cB(\R^\infty)\}} = \cH.
\end{align}
The following spectral theorem is the main result we will need.

\begin{theorem}{\cite[Theorem 4]{samoilenko_spectral_2013}} \label{thm:spectral}
    Let $\{T_j\}_{j=1}^\infty$ be a countable family of CSOs on $\cH$, equipped with a cyclic vector $\Phi_0 \in \cH$ (in the sense of~\eqref{eq:cyclicvec_general}). Then, there exists a probability measure $\mu \in \cM(\R^\infty)$ and a unitary transformation $U: \cH \to L^2(\R^\infty, \mu)$ such that
    \begin{align}
        T_j = U^{-1} \lambda_j U,
    \end{align}
    where $\lambda_j : L^2(\R^\infty, \mu) \to L^2(\R^\infty, \mu)$ is the multiplication operator by the $j^{th}$ independent variable.
\end{theorem}

To apply this, we must show that the $J_{m,i}$ are CSOs and the existence of a cyclic vector. \medskip

\subsubsection{Commutativity and cyclic vectors}
Now, we return to the Jacobi matrices and start by showing conditions for commutativity along with establishing the existence of a cyclic vector. 

\begin{proposition} \label{prop:commutativity_conditions}
    Let $\tbp$ be an orthonormal basis which satisfies the relations in~\Cref{cor:relation_on}. Then for all $n_1 \leq n_2 \in \N$, $m_1, m_2 \in \N$ and $i \in [N_{m_1}], j \in [N_{m_2}]$, we have
    \begin{align} \label{eq:commutativity_conditions}
        &\sum_{r \in R^1_{1,1} } (M^{n_1 - r}_{m_1+r, m_1, i})^\tr M^{n_2-r}_{m_2 + r, m_2, j} + \sum_{r \in R^2_{1,2}} M^{r-n_1}_{n_1+m_1, m_1, i} M^{n_2 - r}_{m_2+r, m_2, j} + \sum_{r \in R^3_{1,2}} M^{r-n_1}_{n_1+m_1, m_1, i} (M^{r-n_2}_{n_2+m_2,m_2, j})^{\tr}\\
         & \quad = \sum_{r \in R^1_{2,1}} (M^{n_1 - r}_{m_2+r, m_2, j})^\tr M^{n_2-r}_{m_1 + r, m_1, i} + \sum_{r \in R^2_{2,1}} M^{r-n_1}_{n_1+m_2, m_2, j} M^{n_2 - r}_{m_1+r, m_1, i} + \sum_{r \in R^3_{2,1}} M^{r-n_1}_{n_1+m_2, m_2, j} (M^{r-n_2}_{n_2+m_1,m_1, i})^{\tr}, \nonumber
    \end{align}
    where
    \begin{align}
        R^1_{i,j} &= [n_1 - m_i, n_1] \cap [n_2 - m_j, n_2], \\
        R^2_{i,j} &= [n_1+1, n_1+m_i] \cap [n_2 - m_j, n_2], \\
        R^3_{i,j} &= [n_1 + 1, n_1 + m_i] \cap [n_2+1, n_2+m_j].
    \end{align}
\end{proposition}
\begin{proof}
    We obtain these relations by applying the relations in~\Cref{cor:relation_on} to the equality
    \begin{align}
        (w_{m_1, i} \tbp_{n_1}, w_{m_2, j} \tbp_{n_2}) = (w_{m_2, j} \tbp_{n_1}, w_{m_1, i} \tbp_{n_2}).
    \end{align}
\end{proof}


\begin{lemma} \label{lem:cyclic_vector}
    The basis vector $\Phi_0 = \phi_{\bzero} \in H(W^\gr)$ is a cyclic vector for $\{J_{m,i}\}$.
\end{lemma}
\begin{proof}
    Let $H_E = \overline{\text{span}\{E(B) \Phi_0 \, : \, B \in \cB(\hW)\}} \subset H(V)$. Then by the spectral description of $J_{m,i}$ in~\Cref{thm:spectral1}, we have $E(B) J_{m,i} = J_{m,i} E(B)$ for any $B \in B \in \cB(\hW)$, so that $J_{m,i} H_E \subset H_E$. Next, by definition of the Jacobi matrices, we have
    \begin{align}
        J_{m,i} \Phi_{n-m} = \sum_{k=0}^{m} M^k_{n,m,i} \Phi_{n-m+k} + \sum_{k=1}^{m} (M^{   k}_{n-k,m,i})^\tr \Phi_{n-m-k},
    \end{align}
    and thus, by the same arguments as~\Cref{prop:gen_inverse_relation}, we have
    \[
        \Phi_{n} = \sum_{m=1}^{n} \sum_{i=1}^{N_m} \left(J_{m,i} D_{n,m,i}^\tr \Phi_{n-m}\right) + \sum_{k= 0}^{n-1} E_{n}^k \Phi_{k}.
    \]
    Thus, this provides a polynomial in $P_n$ in the $J_{m,i}$ such that $\Phi_n = P_n(J_{m,i})\Phi_0$. Then, since the $\Phi_n$ are dense in $H(W^\gr)$, we have $H_E = H(W^\gr)$. 
\end{proof}

\subsubsection{Bounded case}
Next, we move on to showing that the Jacobi matrices $J_{m,i}$ are a commuting self-adjoint family of operators on $H(W^\gr)$. The arguments depend on whether we assume that $J_{m,i}$ is bounded, and we begin with the bounded setting. These results generalize the classical setting, and while we detail aspects of the proof which differ in our setting, we will defer parts of the proof to existing references in situations where the proof is the same. 

\begin{lemma} \label{lem:bounded_J}
    The operator $J_{m,i}$ is bounded if and only if $
    \sup_{n \geq 0} \|M^k_{n,m,i}\|_2 < \infty$ for all $k \in [-m, m]$.
\end{lemma}
\begin{proof}
    The proof is analogous to ~\cite[Lemma 3.4.3]{dunkl_orthogonal_2014}.
\end{proof}

\begin{theorem} \label{thm:measure_compact}
    Let $\tbp = (\tbp_n)_{n=0}^\infty$ be a sequence in $\Sym(W^\gr)$ such that $\tbp_0 = 1$. The following are equivalent.
    \begin{enumerate}
        \item There exists a determinate measure $\mu \in \cM(\hW)$ with compact support such that $\tbp$ is orthonormal with respect to $\mu$.
        \item Statement (2) in~\Cref{thm:favard_positive} holds and in addition, $
    \sup_{n \geq 0} \|M^k_{n,m,i}\|_2 < \infty$ for all $k \in [-m, m]$.
    \end{enumerate}
\end{theorem}
\begin{proof}
    First, if $\mu$ has compact support, then the the multiplication operators $T_{w_{m,i}} : L^2(\hW, \mu) \to L^2(\hW, \mu)$ are bounded. Because these multiplication operators are represented by the Jacobi matrices $J_{m,i}$ in the $\tbp_n$ basis, the result follows from~\Cref{lem:bounded_J}.

    Next, suppose the second statement holds, so by~\Cref{thm:favard_positive}, $\tbp$ are orthonormal with respect to a positive-definite linear functional $\cL$. We define the Jacobi matrices as in~\eqref{eq:jacobi_matrix}, which are bounded by~\Cref{lem:bounded_J}. In the bounded setting, strong commutativity is equivalent to commutativity of operators. We can directly verify that the conditions in~\Cref{prop:commutativity_conditions} are equivalent to $J_{m_1, i}J_{m_2, j} = J_{m_2, j}J_{m_1, i}$; thus the $J_{m,i}$ are strongly commutative. Next, it is clear by definition of the $J_{m,i}$ that they are symmetric, and since they are bounded, they must be self-adjoint. Because there exists a cyclic vector (\Cref{lem:cyclic_vector}), we apply~\Cref{thm:spectral} to obtain a measure $L^2(\hW, \mu)$ such that $J_{m,i}$ are unitarily equivalent to multiplication operators $T_{m,i}$, so that
    \begin{align*}
        \int_{\hW} \tbp_n(w) \tbp_m^\tr(w) d\mu(w) = \langle \tbp_n(T) 1, \tbp_m^\tr(T)1 \rangle_{L^2(\hW, \mu)} = \langle \tbp_n(J) \Phi_0, \tbp_m^\tr(J) \Phi_0 \rangle_{H(W^\gr)} = \langle \Phi_n, \Phi_m^\tr \rangle_{H(W^\gr)}.
    \end{align*}
    Thus, $\tbp$ is orthonormal with respect to $\mu$. The support $S_{m,i}$ of the spectral measure $E_{m,i}$ of $J_{m,i}$ is compact since $J_{m,i}$ is bounded. Thus, the support of $S = \prod_{m,i} S_{m,i}$ is compact by Tychonoff's theorem. 
    Finally, we note that the projection of $\mu$ to any finite dimensional subspace is determinate (since it is compact), and therefore, $\mu$ is determinate by~\cite[Corollary 5.3]{alpay_moment_2015}.
\end{proof}

\subsubsection{Unbounded case}
Next, we consider the case where the Jacobi matrices $J_{m,i}$ are unbounded operators $J_{m,i} : \cD(J_{m,i}) \to H(W^\gr)$, where the domain $\cD(J_{m,i}) \subset H(W^\gr)$ consists of all sequences $f \in H(W^\gr)$ such that $J_{m,i}f \in H(W^\gr)$.

\begin{lemma} \label{lem:unbounded_self_adjoint}
    Suppose
    \begin{align} \label{eq:unbounded_condition}
        \sum_{n=0}^\infty \frac{1}{\fM_{n,m,i}} = \infty \quad \text{where} \quad \fM_{n,m,i} = \sum_{k=1}^m \sum_{r=n+1}^{n+k} \|M_{m+r-k,m,i}^k\|_2.
    \end{align}
    Then, $J_{m,i}$ is self-adjoint. 
\end{lemma}
\begin{proof}
    Let $f,g \in \cD(J_{m,i})$ such that $f = \sum \ba_k^\tr \Phi_k$ and $g = \sum \bb_k^\tr \Phi_k$. We define $\ba_k = \bb_k = 0$ for $k < 0$. We will prove that $J_{m,i}$ is symmetric so that $\langle J_{m,i} f, g\rangle = \langle J_{m,i} g, f \rangle$. By definition of the $J_{m,i}$, we have
    \begin{align}
        \langle J_{m,i}f, g \rangle = \sum_{r=0}^\infty \left( \sum_{k=1}^m \ba_{r-k}^\tr M^k_{m+r-k,m,i} \bb_r + \ba_r^\tr M^0_{m+r,m,i} \bb_r + \sum_{k=1}^m \ba_{r+k}^\tr (M^k_{m+r, m,i})^\tr \bb_r \right).
    \end{align}
    For $n \in \N$, we define the truncation by
    \begin{align} \label{eq:unbounded_truncation}
        S_n(\langle J_{m,i}f, g \rangle) = \sum_{r=0}^n\left( \sum_{k=1}^m \ba_{r-k}^\tr M^k_{m+r-k,m,i} \bb_r + \ba_r^\tr M^0_{m+r,m,i} \bb_r + \sum_{k=1}^m \ba_{r+k}^\tr (M^k_{m+r, m,i})^\tr \bb_r \right),
    \end{align}
    and note that $\langle J_{m,i}f, g \rangle = \lim_{n \to \infty} S_n(\langle J_{m,i}f, g \rangle)$. Then, by direct computation using that $M_{m+r-k,m,i}^k$ is symmetric, we have
    \begin{align}
        S_n(\langle J_{m,i}f, g \rangle) - S_n(\langle J_{m,i} g, f \rangle) = \sum_{k=1}^m \sum_{r=n+1}^{n+k} -\ba_{r -k}^\tr M_{m+r-k,m,i}^k \bb_{r} + \bb_{r -k}^\tr M_{m+r-k,m,i}^k \ba_{r}.
    \end{align}
    Then, we can bound this by
    \begin{align*}
        |S_n(\langle J_{m,i}f, g \rangle) - S_n(\langle J_{m,i} g, f \rangle)| &\leq \sum_{k=1}^m \sum_{r=n+1}^{n+k} \|M_{m+r-k,m,i}^k\|_2 (\|\ba_{r -k}\|_2^2 + \|\bb_{r}\|_2^2 + \|\bb_{r -k}\|_2^2 + \|\ba_{r}\|_2^2)/2 \\
        & \leq\sum_{k=1}^m \sum_{r=n+1}^{n+k} \|M_{m+r-k,m,i}^k\|_2 \left( \sum_{s=n-m+1}^{n+m} \|\ba_s\|_2^2 + \|\bb_s\|_2^2\right).
    \end{align*}
    Then, suppose $|\langle J_{m,i}f, g \rangle - \langle J_{m,i} g, f \rangle| = \delta$.
    Then, for sufficiently large $N$, we have $|S_n(\langle J_{m,i}f, g \rangle) - S_n(\langle J_{m,i} g, f \rangle)| > \delta/2$ for all $n \geq N$, and therefore
    \begin{align}
        \frac{\delta}{2}\sum_{n \geq N} \frac{1}{\fM_{n,m,i}} \leq \sum_{n\geq N} \sum_{s=n-m+1}^{n+m} \|\ba_s\|_2^2 + \|\bb_s\|_2^2
    \leq 2m (\|\ba\|^2 + \|\bb\|^2)  \leq  2m (\|f\|^2 + \|g\|^2) < \infty,
    \end{align}
    so by the hypothesis, $\delta = 0$ and $J_{m,i}$ is symmetric. Then, the fact that $J_{m,i}$ is self-adjoint can be shown in the same way as~\cite[Lemma 1]{xu_unbounded_1993}.
\end{proof}

\begin{lemma} \label{lem:unbounded_commute}
    Suppose~\eqref{eq:unbounded_condition} holds. Then the operators $J_{m,i}$ pairwise strongly commute.
\end{lemma}
\begin{proof}
    The proof is analogous to~\cite[Lemma 3]{xu_unbounded_1993} by applying~\cite[Lemma 9.2]{nelson_analytic_1959} to the matrices $J_{m,i}$. In particular, we know from~\Cref{prop:commutativity_conditions} that the $J_{m,i}$ commute on the dense subspace $\cD = \Sym(W^\gr) \subset H(W^\gr)$.
\end{proof}

\begin{theorem} \label{thm:measure_unbounded}
    Suppose $\tbp = (\tbp_n)_{n=0}^\infty$ is a sequence in $\Sym(W^\gr)$ such that $\tbp_0 = 1$, statement (2) in~\Cref{thm:favard_positive} holds, and~\eqref{eq:unbounded_condition} holds. Then, there exists a determinate measure $\mu \in \cM(\hW)$ such that $\tbp$ is orthonormal with respect to $\mu$. 
\end{theorem}
\begin{proof}
    The proof is the same as~\Cref{thm:measure_compact} aside from determinacy, where~\Cref{lem:unbounded_self_adjoint,lem:unbounded_commute} are used to show that the $J_{m,i}$ are commuting self-adjoint operators. Thus, there exists a measure $\mu \in \cM(\hW)$ such that $\tbp$ is orthonormal with respect to $\mu$.
    
    In order to prove determinacy, consider a finite dimensional subspace $U \subset W$ which is defined by a finite subset of the basis vectors $\{w_{m_k, i_k}\}_{k=0}^K$ of $W$. Let $\mu_U \in \cM(U)$ denote the pushforward of $\mu$ under the projection to $U$. The induced inner product defined on $\Sym(U^\gr)$
    is simply the restriction of the inner product in $\Sym(W^\gr)$. Let $\tbp^U$ be the collection of orthonormal
    polynomials in $\tbp$ which are valued in $U$. Furthermore, both~\Cref{lem:unbounded_self_adjoint,lem:unbounded_commute} hold for $\tbp^U$. Then, by the same arguments as above, we have a measure $\mu'_U \in \cM(U)$ such that $\tbp^U$ is orthonormal with respect to $\mu'_U$. In particular, all multiplication operators corresponding\footnote{Note that the matrices $M^k_{n,m,i}$ in the recurrence relation for $\tbp^U$ are different in general. We let $J^U_{m_k, i_k}$ denote the Jacobi operators for $\tbp^U$.} to $\{J^U_{m_k, i_k}\}_{k=0}^K$ are self-adjoint, so by~\cite[Theorem 14.2]{schmudgen_moment_2017}, $\mu'_U$ is determinate, so $\mu_U = \mu'_U$. Finally, this holds for any finite dimensional $U \subset W$ constructed from a finite set $\{w_{m_k, i_k}\}_{k=0}^K$ of basis elements, $\mu$ is determinate by~\cite[Corollary 5.3]{alpay_moment_2015}.
\end{proof}

\subsection{Measures on \texorpdfstring{$G(V)$}{G(V)}} \label{ssec:measure_GV}

Now, we return to our original setting of shuffle polynomials and set $W = L(V)$ so that $\Sym(W^\gr) \cong \Sh(V^*)$ as
commutative algebras. Recall that in order recover our interpretation of $\Sh(V^*)$ as signature polynomials, we
evaluate polynomials $p \in \Sh(V^*)$ on the signature of paths. The positive definite Favard's theorem
in~\Cref{thm:favard_positive} immediately holds in this setting as it only specifies an inner product on $\Sh(V^*)$.
However, when we consider~\Cref{thm:measure_compact,thm:measure_unbounded}, which shows the existence of a measure $\mu$ on $\hW$, this leads to a distinction which is absent in the classical setting: the measure $\mu$ may not come from a measure on path space.\medskip

First, $\hL(V)$ is equipped with the product topology and consists of free Lie series on $V$~\cite{reutenauer_free_1993}. Then, when $G(V)$ is equipped with the projective topology, the tensor exponential $\exp_{\otimes} : \hL (V) \to G(V)$ is a homeomorphism.
Note that for any $p \in \Sh(V^*)$ and $w \in \hL(V)$, we have
\begin{align}
    p(w) = \langle p, \exp_{\otimes}(w) \rangle.
\end{align}
Thus, given a measure $\mu \in \cM(\hL (V))$, the pushforward $\mu_G = (\exp_\otimes)_*\mu \in \cM(G(V))$ satisfies
\begin{align} \label{eq:LV_GV_measures}
    \int_{\hL(V)} p(w) q(w) d\mu(w) = \int_W \langle p \shuffle q, \exp_\otimes(w) \rangle d\mu(w) = \int_{G(V)} \langle p \shuffle q, x \rangle d\mu_G(x).
\end{align}
Thus, inner products on $\Sym(L(V)^\gr)$ induced by measures on $\hL(V)$ correspond to inner products on $\Sh(V^*)$ induced by measures on $G(V)$. However, this need not imply $\mu_G$ is induced by a measure on path space. In fact, we have the following inclusions of groups
\begin{align}
    \im(S) \subset \tG(V) \subset G(V),
\end{align}
where $S$ denotes the signature, so we must further consider the support of $\mu_G$. While the signature of paths have an infinite radius of convergence, so $\im(S) \subset \tG(V)$, the characterization of the image is an open problem~\cite[Problem 1.12]{HL10}.  Therefore, we restrict our attention to the first step and consider when the measure $\mu_G$ is supported in $\tG(V)$. We start with the bounded case of~\Cref{thm:measure_compact}.

\begin{proposition}
    Suppose $\tbp = (\tbp_n)_{n=0}^\infty$ be a sequence in $\Sym(L(V)^\gr)$ such that $\tbp_0 = 1$. Suppose statement (2) in~\Cref{thm:favard_positive} holds, and in addition the Jacobi matrices $J_{m,i}$ are bounded and satisfy
    \begin{align} \label{eq:lyndon_infinite_roc_assumption}
        \sum_{m,i} \|J_{m,i}\| \lambda^m < \infty \quad \text{for all} \quad \lambda > 0.
    \end{align}
    Then, there exists a determinate measure $\tmu_G \in \cM(\tG(V))$ such that $\tbp$ is orthonormal with respect to $\tmu_G$.
\end{proposition}
\begin{proof}
    Because the Jacobi matrices are bounded,~\Cref{lem:bounded_J} implies that statement (1) of~\Cref{thm:measure_compact} holds, so there exists a compact measure $\mu \in \cM(\hW)$. Recall that this measure is defined by $\mu(B) = \langle E(B) \Phi_0, \Phi_0 \rangle_{H(L(V))}$, where $E$ is the spectral measure with respect to the family $\{J_{m,i}\}$ of Jacobi matrices. Let $E_{m,i}$ denote the 1D spectral measure of the operator $J_{m,i}$, which is supported on the spectrum $\sigma(J_{m,i})$. As $J_{m,i}$ is a bounded, its spectral radius is bounded by $\|J_{m,i}\|$. By~\eqref{eq:lyndon_infinite_roc_assumption}, this implies that the support of the joint spectral measure $\mu$ on $W$ consists of elements $\ba = \sum_{m} \bc_m^\tr \bw_m$, where $\bw_m$ denotes the column vector of degree $m$ basis elements of $W$, such that
    \begin{align}
        \sum_{m} \|\bc_m\| \lambda^m < \infty
    \end{align}
    for all $\lambda > 0$. Let $\exp^n_{\otimes}(\ba)$ denote the degree $n$ component of $\exp_\otimes(\ba)$. This implies that
    \begin{align}
        \sum_{n=0}^\infty \|\exp^n_{\otimes}(\ba)\| \lambda^n  \leq \sum_{n=0}^\infty \exp\left( \sum_{m=0}^\infty \lambda^m \|\bc_m\| \right) < \infty,
    \end{align}
    where we assume that the Lie algebra basis $\ell_{m,i} = \iota(w_{m,i})$ are normalized so that $\|\ell_{m,i}\|_{T(V)} =1$ for all $m,i$. Thus $\tmu =\beta_* \mu \in \cM(\tG(V))$. This also implies that the expectation of $\tmu_G$ will have infinite radius of convergence, so $\tmu_G$ is determinate by~\cite[Theorem 6.1]{characteristic}.
\end{proof}
This condition on the Jacobi matrices imply that the support of $\tmu$ consist of group-like elements whose logarithm (viewed as elements of $T\ps{V}$) itself has infinite radius of convergence. However, a conjecture by Lyons-Sidarova~\cite{lyons_radius_2006}, and recently refined in~\cite{boedihardjo_cartans_2025} states that: the log-signature of a tree-reduced bounded variation path has infinite radius of convergence if and only if it is conjugate to a line segment. If true, this conjecture places extremely strong conditions on a potential measure on path space which induces $\tmu$. \medskip

In order to obtain infinite radius of convergence in the case of unbounded Jacobi operators, we consider a truncated setting where we restrict $L(V)$ to Lie elements of degree at most $M$, denoted $L_M(V)$. In this case, we take $W = L_M(V)$, and we define 
\begin{align}
    G_M(V) = \{ \beta(\bw) \, : \, \bw \in L_M(V)\} \subset \tG(V).
\end{align}
Note that both $\Sym(L_M(V)^\gr)$ and $G_M(V)$ still contain elements of arbitrary degree (in particular, $\Sym(L_M(V)^\gr)$ is still infinite dimensional).



\begin{proposition}
    Suppose $\tbp = (\tbp_n)_{n=0}^\infty$ is a sequence in $\Sym(L_M(V)^\gr)$ such that $\tbp_0 = 1$. Suppose statement (2) in~\Cref{thm:favard_positive} holds, and the matrices satisfy~\eqref{eq:unbounded_condition}. Then, there exists a determinate measure $\tmu_G \in \cM(G_M(V))$ such that $\tbp$ is orthonormal with respect to $\tmu_G$.
\end{proposition}
\begin{proof}
    Applying~\Cref{thm:measure_unbounded} to $W = L_M(V)$, there exists a measure $\mu \in \cM(L_M(V))$ such that $\tbp$ is orthonormal with respect to its induced linear functional. Then, the pushforward by the restricted tensor exponential $\exp_\otimes : L_M(V) \to G_M(V)$ yields a measure $\mu_G \in \cM(G_M(V))$. As the expectation of $\mu_G$ has infinite radius of convergence in $G(V)$, the measure $\mu_G$ is determinate by~\cite[Theorem 6.1]{characteristic}.
\end{proof}

In fact, when we further truncate the total degree of polynomials considered, we can take advantage of Chow's theorem~\cite{chow_uber_1940}, which shows the surjectivity of the truncated path signature. This leads to a measure on path space which induces the inner product up to polynomials of the truncated total order.

\begin{theorem}
    Suppose $\tbp = (\tbp_n)_{n=0}^\infty$ is a sequence in $\Sym(L_M(V)^\gr)$ such that $\tbp_0 = 1$. Suppose statement (2) in~\Cref{thm:favard_positive} holds, and the matrices satisfy~\eqref{eq:unbounded_condition}. Then, there exists a measure $\rho_M \in \cM(C^{1-\text{var}}([0,T], V))$ such that $(\tbp_n)_{n=0}^M$ is orthonormal with respect to $S_*\rho_M \in \cM(\tG(V))$.
\end{theorem}
\begin{proof}
    Applying~\Cref{thm:measure_unbounded} to $W = L_M(V)$, there exists a measure $\mu \in \cM(L_M(V))$ such that $\tbp$ is orthonormal with respect to its induced linear functional. Then, by the generalized Tchakaloff's theorem~\cite[Theorem 2]{bayer_proof_2006}, there exists a cubature measure $\mu_M = \sum_{i=1}^N \lambda_i \delta_{\ell_i} \in \cM(L_M(V))$ of order $M$ with respect to generators with inhomogeneous grading, where $N \in \N, \lambda_i \in \R$, and $\ell_i \in L_M(V)$. In particular, $(\tbp_n)_{n=0}^M$ is orthonormal with respect to $\mu_M$. Then, by Chow's theorem~\cite{chow_uber_1940} (see also~\cite[Theorem 7.28]{FV10}), there exists a piecewise linear path $X_i \in C^{1-\text{var}}([0,T], V)$ such that $\log(S(X_i)) = \ell_i$. Let $\rho_M = \sum_{i=1}^N \lambda_i \delta_{\ell_i} \in \cM(C^{1-\text{var}}([0,T], V))$ Thus, $(\tbp_n)_{n=0}^M$ is also orthonormal with respect to $S_* \rho_M$.
\end{proof}

\section{Orthogonal polynomials on Wiener space}\label{sec:brownian}

\subsection{The non-time-augmented case: non-existence}
\label{sec:nonexist}

Consider a centered real-valued Gaussian random variable with variance \(\sigma^2\).
Hermite polynomials are the orthogonal family associated with the measure
\(\mu(x)=\frac{1}{\sqrt{2\pi\sigma^2}}\mathrm{e}^{-\frac{1}{2\sigma^2}x^2}\), i.e., a sequence of polynomials
\(H_n(x)\), \(n\ge 0\) such that
\[
  \int H_n(x)H_m(x)\,\mathrm{d}\mu(x) = n!\delta_{n,m}.
\]

They may be characterized by their generating function
\[
  G(t, x) = \exp\left( tx - \frac12\sigma^2t^2 \right) = \sum_{n=0}^\infty\frac{1}{n!}H_n(x)t^n.
\]

The first few are \(H_0(x) = 1\), \(H_1(x) = x\), \(H_2(x) = x^2 - \sigma^2\), \(H_3(x)=x^3-3\sigma^2x\), and in general
\[
  H_n(x) = n!\sum_{k=0}^{\lfloor n/2\rfloor}\frac{\sigma^{2k}}{k!(n-2k)!2^k}x^{n-2k}.
\]

Fix a vector space \(V\) and let \(\boldsymbol{X}\) be a centered Gaussian vector with covariance matrix \(\Sigma\),
which we regard as a quadratic form on \(V^*\).
Multivariate Hermite polynomials are the block-orthogonal family associated with the pushforward measure
\(\mu=\boldsymbol{X}_*\mathbb{P}=\mathbb{P}\circ\boldsymbol{X}^{-1}\).
They may be characterized by their generating function
\begin{equation}\label{eq:hermite.gen}
  G(t\lambda,\boldsymbol{x})\coloneqq\exp\left(t\lambda(\boldsymbol{x})-\frac{t^2}{2}\Sigma(\lambda)\right)=\sum_{n=0}^{\infty}\frac{t^n}{n!}H_n(\lambda,\boldsymbol{x})
\end{equation}

Since for a given \(\lambda\in V^*\) the random variable \(\lambda(\boldsymbol{X})\) is Gaussian with variance
\(\Sigma(\lambda)\) it follows that the polynomials are the same univariate ones as above, evaluated at \(\lambda(\boldsymbol{x})\) with variance \(\Sigma(\lambda)\).

From the identity \(\mathbb{E}[\mathrm{e}^{\lambda(\boldsymbol{X)}}]=\mathrm{e}^{\frac12\Sigma(\lambda)}\) it follows that
\[
  \mathbb{E}[G(s\lambda,\boldsymbol{X})G(t\lambda',\boldsymbol{X})] = \mathrm{e}^{\frac12st B(\lambda,\lambda')}
\]
where \(B(\lambda,\lambda') = \frac12(\Sigma(\lambda+\lambda')-\Sigma(\lambda)-\Sigma(\lambda'))\) is the associated
bilinear form.
From this we see that
\[
  \int H_n(\lambda,\boldsymbol{x}) H_m(\lambda',\boldsymbol{x})\,\mathrm{d}\mu(\boldsymbol{x}) =
  n!B(\lambda,\lambda')^n\delta_{n,m}.
\]
 
It should also be clear that \(H_n(\lambda,\boldsymbol{x})\) depends polynomially on both \(\lambda\) and
\(\boldsymbol{x}\), homogeneously so on \(\lambda\) but not on \(x\).
For example, \(H_0(\lambda,\boldsymbol{x})=1, H_1(\lambda,\boldsymbol{x})=\lambda_ix^i\) and 
\[
  H_2(\lambda,\boldsymbol{x}) = \lambda_i\lambda_j(x^ix^j - \Sigma^{ij}),
\]
and so on, where the Einstein summation convention is in place.
We note that, although these expressions depend on a choice of basis for \(V\) and the corresponding dual basis for
\(V^*\), the polynomials \(H_n\) are independent of this.

Indeed, we have intentionally omitted any reference to the dimension of \(\boldsymbol{X}\) since it is evident from
\eqref{eq:hermite.gen} that the expression for \(H_n\) does not depend on it.
In fact, \eqref{eq:hermite.gen} makes sense on any finite-dimensional space carrying a fixed quadratic form \(\Sigma\) and it is independent of any choice apart from this data.
Multivariate Hermite polynomials are natural in a categorical sense: consider the category \(\mathsf{Quad}\) of
\emph{quadratic spaces}, i.e., pairs \((V, \Sigma_V)\) where \(V\) is a finite-dimensional vector space and \(\Sigma\) is
a quadratic form. Maps between quadratic spaces are \emph{injective isometries} (and not necessarily surjective), that is, linear maps \(\varphi\colon U\to V\) such that \(\Sigma_V(\varphi u)=\Sigma_U(u)\).

Now, let \((U^*,\Sigma_U)\) and \((V^*,\Sigma_V)\) be nondegenerate quadratic spaces and \(\varphi\colon U\to V\) be such that 
\(\varphi^*\colon V^*\to U^*\), \(\varphi^*\lambda\mapsto\lambda\circ\varphi\) is an isometry.
For fixed \(\lambda\in U^*\), the map $H_n(\lambda(\cdot);\Sigma(\lambda))$ sends a monomial in \(U\) to a polynomial, so
it lies in $\Sym(U)^*$. By homogeneity of \(H_n\) in \(\lambda\) this map is then \(H_n\colon
\Sym_n(U^*)\to \Sym(U)^*\).
The identity
\[
  G_V(\varphi^*\lambda,\boldsymbol{x}) = \exp\left( \lambda(\varphi\boldsymbol{x}) - \frac12\Sigma_V(\varphi^*\lambda) \right) = G_U(\lambda,\varphi\boldsymbol{x})
\]
implies \(H_n(\varphi^*\lambda,\cdot) = H_n(\lambda,\varphi\cdot) = \varphi^*H_n(\lambda,\cdot)\) for all \(n\ge 0\), that is, the diagram
\[
\begin{tikzcd}
  \Sym(V^*)\dar["\Sym(\varphi^*)"]\rar["H_V"]&\Sym(V)^*\dar["\Sym(\varphi)^*"]\\
  \Sym(U^*)\rar["H_U"]&\Sym(U)^*
\end{tikzcd}
\]
commutes.
By making use of the isomorphism \(V\cong V^*\) induced by the nondegenerate bilinear form \(B_V\) associated to
\(\Sigma_V\) we obtain a commuting diagram
\[
\begin{tikzcd}
  \Sym(V^*)\dar["\Sym(\varphi^*)"]\rar["H_V"]&\Sym(V^*)\dar["\Sym(\varphi^*)"]\\
  \Sym(U^*)\rar["H_U"]&\Sym(U^*),
\end{tikzcd}
\]
This property has the advantage, already hinted above, of making any computation involving these polynomials completely
independent of any choice but the covariance operator \(\Sigma\).

Now, we would like a similar map for the non-time-augmented Stratonovich signature of Brownian motion.
In the more restrictive setting where morphisms are taken to be bijective isometries (so that the dimension of \(U\) is
fixed), this is possible, the map being given by Gram-Schmidt block orthogonalization.
Unfortunately, in the more interesting case of morphisms being (not necessarily bijective) isometries this is not possible beyond degree 4 as we now show.

We turn to the $d$-dimensional Wiener measure, without drift for the time being. Switching notation slightly to that in
\Cref{sec:L2}, we recall the expression for the expected signature of Brownian motion \cite{fawcett, LV04}, sometimes called Fawcett's formula, in any given
basis of \(\mathbb{R}^d\),
\begin{equation}\label{eq:fawcett}
\mathbb E S(W)_{0,T} = \exp\! \bigg( \frac T2 \sum_{\gamma = 1}^d \gamma\gamma\bigg) = \sum_{n = 0}^\infty \frac{T^n}{2^n n!}\sum_{\gamma_1,\ldots, \gamma_n = 1}^d \gamma_1\gamma_1 \ldots \gamma_n\gamma_n .
\end{equation}
For simplicty we set \(T=1\) for the remainder of this section.
We would like to block-orthogonalise words in $\mathrm{Sh}(\mathbb R^d)$ w.r.t.\ this measure. Following our considerations for Hermite, we make the following definition of \emph{natural maps}.

\begin{definition}\label{def:natural}
  We call a family of maps \(H_V\colon T(V)\to T(V)\) \emph{natural} if for any injective isometry \(\varphi\colon V\to W\) the diagram
  \begin{center}
    \begin{tikzcd}
      T(V)\rar{H_V}\dar{T(\varphi)}&T(V)\dar{T(\varphi)}\\
      T(W)\rar{H_W}&T(W)
    \end{tikzcd}
  \end{center}
  commutes, where \(T(\varphi)\colon T(V)\to T(W)\) is the induced map \(T(\varphi)(v_1\dotsm
  v_n)=\varphi(v_1)\dotsm\varphi(v_n)\).
\end{definition}

\begin{remark}[Restriction property]\label{def:nice}
%
In particular, if \(\varphi\colon\mathbb{R}^d\to\mathbb{R}^D\) with \(d<D\) is the inclusion map, the orthogonalization
map \(H_{\mathbb{R}^D}\) restricts to \(H_{\mathbb{R}^d}\) on words only containing the first \(d\) coordinates.
\end{remark}

The following is the main result of this subsection.

\begin{theorem}\label{thm:non_natural}
    For $V=\R^d$, define $H_V\colon T(V)\to T(V)$ by $H_V(w) = p_w$ for every word $w$, where $p_w$ is the block-orthogonalisation endomorphism defined in \eqref{eq:p} with respect to $d$-dimensional Brownian motion.
    Then $H_V$ is not natural for $d\geq 2$.
\end{theorem}

 \begin{proof}
 Using the \texttt{Wolfram} code listed in \Cref{ap:code}, we obtain for \(d=3\) and \(w=\texttt{11112}\), the image
\[
  p_w=\texttt{11112}+\frac{1}{96}\,\texttt{332}-\frac{1}{96}\,\texttt{233}-\frac{1}{96}\,\texttt{211}-\frac{35}{96}\,\texttt{112}+\frac{5}{96}\,\texttt{2}.
\]
Since \(p_w\) depends on letters not in \(w\) the map cannot be natural.
For an example with \(d=2\), and the same \(w\) we obtain
\[
  p_w=\texttt{11112}-\frac{1}{80}\,\texttt{211}-\frac{29}{80}\,\texttt{112}+\frac{5}{96}\,\texttt{2}
\]
and we see that the coefficients have changed.
 \end{proof}

While the above proof is short (following the computational effort in \Cref{ap:code}), it perhaps does not reveal the mechanism behind the failure of naturality.
We now provide an alternative, less computational, explanation for this behaviour, albeit without providing all the details.

While the map \(w\mapsto p_w\) is completely determined by \eqref{eq:p}, computing the coefficients in this expression
involves solving a linear system which might be difficult and not very insightful. In particular, it is difficult to
asses the naturality of this map in its current form.
Therefore, we attempt to come up with a procedure for computing it
that takes the structure of the shuffle algebra and Fawcett's formula \eqref{eq:fawcett} into account,
i.e. by using the formula
\begin{equation}\label{eq:shuffle_Fawcett}
 (u,w) = \scal{u\shuffle w , \E S(W)_{0,1}}
\;.   
\end{equation}
For this we develop some graphical notation. Each dot will be a placeholder for a letter in the alphabet $\{1,\ldots,d\}$, and each arc will stand for a contraction, i.e. we replace the subword $(\alpha,\beta)$ formed by the contraction by a Kronecker delta $\delta_{\alpha\beta}$.
Replacing indices with degree, the first three projections
(which do satisfy the restriction property) can be written as
\begin{equation}\label{eq:diagEx}
e_0 = \varnothing, \quad e_1 = \tikz[baseline = -0.1ex]{\draw[fill] (0,0) circle [radius=0.04];}\ ,\quad
e_2 = \tikz[baseline = -0.1ex]{
\draw[fill] (0,0) circle [radius=0.04];
\draw[fill] (0.3,0) circle [radius=0.04];
}\ - \frac 12\
\tikz[baseline = -0.1ex]{
\draw[fill] (0,0) circle [radius=0.04];
\draw[fill] (0.3,0) circle [radius=0.04];
\draw (0,0) .. controls (0,0.3) and (0.3,0.3) .. (0.3,0);
}\ ,\quad
e_3 = \tikz[baseline = -0.1ex]{
\draw[fill] (0,0) circle [radius=0.04];
\draw[fill] (0.3,0) circle [radius=0.04];
\draw[fill] (0.6,0) circle [radius=0.04];
}\ - \frac 14(\
\tikz[baseline = -0.1ex]{
\draw[fill] (0,0) circle [radius=0.04];
\draw[fill] (0.3,0) circle [radius=0.04];
\draw[fill] (0.6,0) circle [radius=0.04];
\draw (0,0) .. controls (0,0.3) and (0.3,0.3) .. (0.3,0);
}\ + \ \tikz[baseline = -0.1ex]{
\draw[fill] (0,0) circle [radius=0.04];
\draw[fill] (0.3,0) circle [radius=0.04];
\draw[fill] (0.6,0) circle [radius=0.04];
\draw (0.3,0) .. controls (0.3,0.3) and (0.6,0.3) .. (0.6,0);
}\ ) \text{.}
\end{equation}
Note $e_n$ on its own does not identify an element of the shuffle algebra: this is only true once $n$ letters are supplied.
More precisely, for a word $w= \alpha_1\ldots\alpha_n$, $p_w$ is given by $e_n$ where we assign the letters $\alpha_1,\ldots,\alpha_n$ to the $n$ dots.

We make the following general ansatz for $e_n$.
Since $e_n$ must be monic, we start with the string of $n$ dots (the word itself). We want to impose $(e_n, e_m) = 0$ for all $m < n$, and it is enough to do this for $m \equiv n \pmod{2}$. Taking the inner product as in \eqref{eq:shuffle_Fawcett}
of the word itself with any word of degree $n - 2$,
by the formula \eqref{eq:fawcett},
will leave some arcs between consecutive nodes, so it makes sense to add some multiple of such elements, where the multiple will be solved for. Taking the inner product of this element with elements of degree $n - 4$ will yield diagrams with two arcs, each of which between consecutive nodes; this is not all, however: it will also diagrams containing nested arcs, since the single arcs included in the preceding step are \say{skipped over}. Continuing recursively in this manner suggests the following form for $e_n$: begin with a string of $n$ nodes with no pairings, and add unknown multiples of diagrams in which the pairings obey the following rules:
\begin{itemize}
\item The pairing is non-crossing, i.e.\ no two arcs intersect;
\item Each node under an arc is itself paired.
\end{itemize}
This means in general, our diagrams will consist of \say{islands} of nodes that are paired alternated with strings of
unpaired nodes. Similar diagrams have appeared in \cite{anshelevich_appell}.
At degree $4$ this becomes
\[
e_4 = \ \tikz[baseline = -0.1ex]{
\draw[fill] (0,0) circle [radius=0.04];
\draw[fill] (0.3,0) circle [radius=0.04];
\draw[fill] (0.6,0) circle [radius=0.04];
\draw[fill] (0.9,0) circle [radius=0.04];
} \ + x_1 ( \ \tikz[baseline = -0.1ex]{
\draw[fill] (0,0) circle [radius=0.04];
\draw[fill] (0.3,0) circle [radius=0.04];
\draw[fill] (0.6,0) circle [radius=0.04];
\draw[fill] (0.9,0) circle [radius=0.04];
\draw (0,0) .. controls (0,0.3) and (0.3,0.3) .. (0.3,0);
} \ ) + x_2( \ \tikz[baseline = -0.1ex]{
\draw[fill] (0,0) circle [radius=0.04];
\draw[fill] (0.3,0) circle [radius=0.04];
\draw[fill] (0.6,0) circle [radius=0.04];
\draw[fill] (0.9,0) circle [radius=0.04];
\draw (0.3,0) .. controls (0.3,0.3) and (0.6,0.3) .. (0.6,0);
} \ ) + x_3( \ \tikz[baseline = -0.1ex]{
\draw[fill] (0,0) circle [radius=0.04];
\draw[fill] (0.3,0) circle [radius=0.04];
\draw[fill] (0.6,0) circle [radius=0.04];
\draw[fill] (0.9,0) circle [radius=0.04];
\draw (0.6,0) .. controls (0.6,0.3) and (0.9,0.3) .. (0.9,0);
} \ ) + y_1( \ \tikz[baseline = -0.1ex]{
\draw[fill] (0,0) circle [radius=0.04];
\draw[fill] (0.3,0) circle [radius=0.04];
\draw[fill] (0.6,0) circle [radius=0.04];
\draw[fill] (0.9,0) circle [radius=0.04];
\draw (0,0) .. controls (0,0.3) and (0.3,0.3) .. (0.3,0);
\draw (0.6,0) .. controls (0.6,0.3) and (0.9,0.3) .. (0.9,0);
} \ )
+ y_2( \ \tikz[baseline = -0.1ex]{
\draw[fill] (0,0) circle [radius=0.04];
\draw[fill] (0.3,0) circle [radius=0.04];
\draw[fill] (0.6,0) circle [radius=0.04];
\draw[fill] (0.9,0) circle [radius=0.04];
\draw (0.3,0) .. controls (0.3,0.2) and (0.6,0.2) .. (0.6,0);
\draw (0.0,0) .. controls (0.0,0.4) and (0.9,0.4) .. (0.9,0);
} \ ) .
\]
Solving for $(e_4, e_2) = 0 = (e_4, e_0)$ we obtain a unique solution!
\begin{equation}\label{eq:level4}
e_4 = \ \tikz[baseline = -0.1ex]{
\draw[fill] (0,0) circle [radius=0.04];
\draw[fill] (0.3,0) circle [radius=0.04];
\draw[fill] (0.6,0) circle [radius=0.04];
\draw[fill] (0.9,0) circle [radius=0.04];
} \ -\frac 16 ( \ \tikz[baseline = -0.1ex]{
\draw[fill] (0,0) circle [radius=0.04];
\draw[fill] (0.3,0) circle [radius=0.04];
\draw[fill] (0.6,0) circle [radius=0.04];
\draw[fill] (0.9,0) circle [radius=0.04];
\draw (0,0) .. controls (0,0.3) and (0.3,0.3) .. (0.3,0);
} \ ) -\frac 16( \ \tikz[baseline = -0.1ex]{
\draw[fill] (0,0) circle [radius=0.04];
\draw[fill] (0.3,0) circle [radius=0.04];
\draw[fill] (0.6,0) circle [radius=0.04];
\draw[fill] (0.9,0) circle [radius=0.04];
\draw (0.3,0) .. controls (0.3,0.3) and (0.6,0.3) .. (0.6,0);
} \ ) -\frac 16( \ \tikz[baseline = -0.1ex]{
\draw[fill] (0,0) circle [radius=0.04];
\draw[fill] (0.3,0) circle [radius=0.04];
\draw[fill] (0.6,0) circle [radius=0.04];
\draw[fill] (0.9,0) circle [radius=0.04];
\draw (0.6,0) .. controls (0.6,0.3) and (0.9,0.3) .. (0.9,0);
} \ ) + \frac{1}{24}( \ \tikz[baseline = -0.1ex]{
\draw[fill] (0,0) circle [radius=0.04];
\draw[fill] (0.3,0) circle [radius=0.04];
\draw[fill] (0.6,0) circle [radius=0.04];
\draw[fill] (0.9,0) circle [radius=0.04];
\draw (0,0) .. controls (0,0.3) and (0.3,0.3) .. (0.3,0);
\draw (0.6,0) .. controls (0.6,0.3) and (0.9,0.3) .. (0.9,0);
} \ )
+ \frac{1}{12}( \ \tikz[baseline = -0.1ex]{
\draw[fill] (0,0) circle [radius=0.04];
\draw[fill] (0.3,0) circle [radius=0.04];
\draw[fill] (0.6,0) circle [radius=0.04];
\draw[fill] (0.9,0) circle [radius=0.04];
\draw (0.3,0) .. controls (0.3,0.2) and (0.6,0.2) .. (0.6,0);
\draw (0.0,0) .. controls (0.0,0.4) and (0.9,0.4) .. (0.9,0);
} \ ) .
\end{equation}
Now we ask whether this ansatz works in general. At level 5 it reads
\begin{equation}\label{eq:e5}
\begin{split}
e_5 ={}& 
\ \tikz[baseline = -0.1ex]{
\draw[fill] (0,0) circle [radius=0.04];
\draw[fill] (0.3,0) circle [radius=0.04];
\draw[fill] (0.6,0) circle [radius=0.04];
\draw[fill] (0.9,0) circle [radius=0.04];
\draw[fill] (1.2,0) circle [radius=0.04];} \ + x_1( \ \tikz[baseline = -0.1ex]{
\draw[fill] (0,0) circle [radius=0.04];
\draw[fill] (0.3,0) circle [radius=0.04];
\draw[fill] (0.6,0) circle [radius=0.04];
\draw[fill] (0.9,0) circle [radius=0.04];
\draw[fill] (1.2,0) circle [radius=0.04];
\draw (0,0) .. controls (0,0.3) and (0.3,0.3) .. (0.3,0);
} \ ) + x_2( \ \tikz[baseline = -0.1ex]{
\draw[fill] (0,0) circle [radius=0.04];
\draw[fill] (0.3,0) circle [radius=0.04];
\draw[fill] (0.6,0) circle [radius=0.04];
\draw[fill] (0.9,0) circle [radius=0.04];
\draw[fill] (1.2,0) circle [radius=0.04];
\draw (0.3,0) .. controls (0.3,0.3) and (0.6,0.3) .. (0.6,0);
} \ ) + x_3( \ \tikz[baseline = -0.1ex]{
\draw[fill] (0,0) circle [radius=0.04];
\draw[fill] (0.3,0) circle [radius=0.04];
\draw[fill] (0.6,0) circle [radius=0.04];
\draw[fill] (0.9,0) circle [radius=0.04];
\draw[fill] (1.2,0) circle [radius=0.04];
\draw (0.6,0) .. controls (0.6,0.3) and (0.9,0.3) .. (0.9,0);
} \ )+ x_4( \ \tikz[baseline = -0.1ex]{
\draw[fill] (0,0) circle [radius=0.04];
\draw[fill] (0.3,0) circle [radius=0.04];
\draw[fill] (0.6,0) circle [radius=0.04];
\draw[fill] (0.9,0) circle [radius=0.04];
\draw[fill] (1.2,0) circle [radius=0.04];
\draw (0.9,0) .. controls (0.9,0.3) and (1.2,0.3) .. (1.2,0);
} \ )\\
&+ y_1( \ \tikz[baseline = -0.1ex]{
\draw[fill] (0,0) circle [radius=0.04];
\draw[fill] (0.3,0) circle [radius=0.04];
\draw[fill] (0.6,0) circle [radius=0.04];
\draw[fill] (0.9,0) circle [radius=0.04];
\draw[fill] (1.2,0) circle [radius=0.04];
\draw (0,0) .. controls (0,0.3) and (0.3,0.3) .. (0.3,0);
\draw (0.6,0) .. controls (0.6,0.3) and (0.9,0.3) .. (0.9,0);
} \ ) + y_2( \ \tikz[baseline = -0.1ex]{
\draw[fill] (0,0) circle [radius=0.04];
\draw[fill] (0.3,0) circle [radius=0.04];
\draw[fill] (0.6,0) circle [radius=0.04];
\draw[fill] (0.9,0) circle [radius=0.04];
\draw[fill] (1.2,0) circle [radius=0.04];
\draw (0,0) .. controls (0,0.3) and (0.3,0.3) .. (0.3,0);
\draw (0.9,0) .. controls (0.9,0.3) and (1.2,0.3) .. (1.2,0);
} \ ) + y_3( \ \tikz[baseline = -0.1ex]{
\draw[fill] (0,0) circle [radius=0.04];
\draw[fill] (0.3,0) circle [radius=0.04];
\draw[fill] (0.6,0) circle [radius=0.04];
\draw[fill] (0.9,0) circle [radius=0.04];
\draw[fill] (1.2,0) circle [radius=0.04];
\draw (0.3,0) .. controls (0.3,0.3) and (0.6,0.3) .. (0.6,0);
\draw (0.9,0) .. controls (0.9,0.3) and (1.2,0.3) .. (1.2,0);
} \ ) + y_4( \ \tikz[baseline = -0.1ex]{
\draw[fill] (0,0) circle [radius=0.04];
\draw[fill] (0.3,0) circle [radius=0.04];
\draw[fill] (0.6,0) circle [radius=0.04];
\draw[fill] (0.9,0) circle [radius=0.04];
\draw[fill] (1.2,0) circle [radius=0.04];
\draw (0.0,0) .. controls (0.0,0.4) and (0.9,0.4) .. (0.9,0);
\draw (0.3,0) .. controls (0.3,0.2) and (0.6,0.2) .. (0.6,0);
} \ ) + y_5( \ \tikz[baseline = -0.1ex]{
\draw[fill] (0,0) circle [radius=0.04];
\draw[fill] (0.3,0) circle [radius=0.04];
\draw[fill] (0.6,0) circle [radius=0.04];
\draw[fill] (0.9,0) circle [radius=0.04];
\draw[fill] (1.2,0) circle [radius=0.04];
\draw (0.3,0) .. controls (0.3,0.4) and (1.2,0.4) .. (1.2,0);
\draw (0.6,0) .. controls (0.6,0.2) and (0.9,0.2) .. (0.9,0);
} \ ) .
\end{split}
\end{equation}
Now we claim that the linear system arising from the constraints $(e_5, e_3) = 0 = (e_5, e_1)$ does not have a solution.
Indeed, the constraints are obviously equivalent to $(e_5, w_3) = 0 = (e_5, w_1)$, where $w_n$ denotes a word of length $n$.

On the one hand, the equation \((e_5,w_1)=0\) yields
\begin{equation}\label{eq:w1}
  0 = \dotsb + \left(y_5+\frac14 x_3\right)( \ \tikz[baseline = -0.1ex]{
\draw[fill] (0,0) circle [radius=0.04];
\draw[fill] (0.3,0) circle [radius=0.04];
\draw[fill] (0.6,0) circle [radius=0.04];
\draw[fill] (0.9,0) circle [radius=0.04];
\draw[fill] (1.2,0) circle [radius=0.04];
\draw[fill, red] (1.5,0) circle [radius=0.04];
\draw (0,0) .. controls (0,0.6) and (1.5,0.6) .. (1.5,0);
\draw (0.3,0) .. controls (0.3,0.4) and (1.2,0.4) .. (1.2,0);
\draw (0.6,0) .. controls (0.6,0.2) and (0.9, 0.2) .. (0.9,0);
} \ ) + \dotsb
\end{equation}
where $\cdots$ represents terms with different full pairings and the red dot indicates the letter not in the original word.
Taking an alphabet with at least three letters and choosing the dot assignments as $\alpha\beta\gamma\gamma\beta\alpha$ for the six dots in \eqref{eq:w1},
we see that all terms in $\cdots$ vanish and the constraints imply \(y_5+\tfrac14 x_3=0\).

On the other hand, the equation \((e_5,w_3)=0\) gives
\[
  0 = \left(\frac{1}{48}+\frac16 x_3\right)( \ \tikz[baseline = -0.1ex]{
\draw[fill] (0,0) circle [radius=0.04];
\draw[fill] (0.3,0) circle [radius=0.04];
\draw[fill] (0.6,0) circle [radius=0.04];
\draw[fill] (0.9,0) circle [radius=0.04];
\draw[fill] (1.2,0) circle [radius=0.04];
\draw[fill, red] (1.5,0) circle [radius=0.04];
\draw[fill, red] (1.8,0) circle [radius=0.04];
\draw[fill, red] (2.1,0) circle [radius=0.04];
\draw (0,0) .. controls (0,0.6) and (1.5,0.6) .. (1.5,0);
\draw (0.3,0) .. controls (0.3,0.4) and (1.8,0.4) .. (1.8,0);
\draw (0.6,0) .. controls (0.6,0.2) and (0.9, 0.2) .. (0.9,0);
\draw (1.2,0) .. controls (1.2,0.3) and (2.1,0.3) .. (2.1,0);
} \ ) + \dotsb + \left(\frac14 y_5+\frac{1}{12} x_3\right)( \ \tikz[baseline = -0.1ex]{
\draw[fill] (0,0) circle [radius=0.04];
\draw[fill] (0.3,0) circle [radius=0.04];
\draw[fill] (0.6,0) circle [radius=0.04];
\draw[fill] (0.9,0) circle [radius=0.04];
\draw[fill] (1.2,0) circle [radius=0.04];
\draw[fill, red] (1.5,0) circle [radius=0.04];
\draw[fill, red] (1.8,0) circle [radius=0.04];
\draw[fill, red] (2.1,0) circle [radius=0.04];
\draw (0,0) .. controls (0,0.6) and (1.5,0.6) .. (1.5,0);
\draw (0.3,0) .. controls (0.3,0.4) and (1.2,0.4) .. (1.2,0);
\draw (0.6,0) .. controls (0.6,0.2) and (0.9, 0.2) .. (0.9,0);
\draw (1.8,0) .. controls (1.8,0.3) and (2.1,0.3) .. (2.1,0);
} \ ) + \dotsb
\;.
\]
It can be checked that these are the only terms producing the corresponding diagrams.
The difference in the coefficients comes from the fact that, in each cases, there are different numbers of shuffles in \eqref{eq:shuffle_Fawcett} that yield the diagram, and also the number of contractions is different.
This yields \(y_5+\tfrac13 x_3=0\) and \(\frac18+x_3=0\)
which contradicts the earlier equation \(y_5+\tfrac14 x_3=0\).

This proves that the general ansatz as in \eqref{eq:e5} is not possible.
What is missing to make the above argument an alternative proof of
\Cref{thm:non_natural} is a justification for why naturality would imply that $p_w$ must be given by $e_n$ as in the above ansatz.

Every natural transformation in the sense of \Cref{def:natural} must be \(O(V)\)-equivariant.
The later class of maps is indexed by Brauer diagrams \cite{brauer} (see
also the more modern article \cite{lehrer_brauer_2015}).
Roughly speaking, there are four types of basic diagrams (maps) \cite[Theorem 2.6]{lehrer_brauer_2015} yielding all
possible transformations by composition and tensorization.
These are the identity map, transposition of two tensor factors; \emph{caps}, corresponding to contractions
\(\alpha\beta\mapsto\delta_{\alpha\beta}\) as above and \emph{cups}, corresponding to maps of the form
\[
  \mathbb{R}\ni\lambda\mapsto\lambda\sum_{\alpha=1}^d\alpha\alpha,
\]
where \(d\) denotes the dimension of the underlying vector space.
Note that of these four, only cups depend on the dimension and are therefore not natural, hence our ansatzes
in \eqref{eq:diagEx}-\eqref{eq:e5} include only identities and caps.
The choice of including only noncrossing diagrams is motivated by Fawcett's formula,
but we do not give a non-computational proof that crossing diagrams do not appear.

We have, however, analysed computationally the same ansatz as \eqref{eq:diagEx}-\eqref{eq:e5} where we now include
all possible pairings up to level 5.
This extended ansatz yields the same solutions up to degree four: the extra degrees of freedom are set to zero, and at
degree 5 the system still contains a subset of inconsistent equations. See \Cref{ap:code} for Wolfram code listings
exploring this.

\subsection{The time-augmented case: It\^o orthogonal polynomials}\label{subsec:ito}

We have seen that, without including time, the block orthogonalisation
with respect to the signature of Brownian motion is not natural in the sense of \Cref{thm:non_natural}.
We now show that including time as a coordinate solves this problem, as it enables the use of the It\^o integral, whose product rule involves quadratic variation corrections.
Let $V$ be an Euclidean vector space, for which we temporarily fix an orthonormal frame, i.e.\ an isomorphism $V \cong \mathbb R^d$; we will return to coordinate-free considerations later on, and recall that even without the choice of a basis the inner product yields a canonical identification $V \cong V^*$.
Write $\widetilde V \coloneqq \mathbb R \oplus V$.
Elements of $T(\widetilde V)$ can be identified with linear combinations of words in the alphabet $[d]_\zero \coloneqq \{\zero,1,\ldots,d\}$; write $[d]^\bullet_\zero$ for the set of such words. Note the distinction between the letter $\zero$ and the empty word $\varnothing$, which spans $\mathbb R = \widetilde V^{\otimes 0}$. Recall the \emph{quasi-shuffle} product on $T(\widetilde V)$, given recursively on words as follows and extended bilinearly:
\begin{equation}\label{eq:qshprod}
u\alpha \qshuffle v\beta = (u \qshuffle v\beta)\alpha + (u\alpha \qshuffle v)\beta + (u \qshuffle v)[\alpha,\beta] \qquad
\text{with}\quad [\alpha,\beta] = \one^{\alpha \neq \zero}_{\alpha\beta} \zero .
\end{equation}
Denote $\widehat{\mathrm{Sh}}(\widetilde V) = T(\widetilde V)$ endowed with this product. We consider two gradings on this space: $\# w$ denotes the number of letters in the word $w$, and
\begin{equation}
|w| \coloneqq \# w + \text{number of \zero's in $w$.}
\end{equation}
Given a $V$-valued Wiener process $B$, write $\widetilde B_t \coloneqq (t, B_t)$ for its time augmentation in the $0^\text{th}$ coordinate. We denote $\widehat S(\widetilde B)$ its It\^o signature, i.e.\ \eqref{eq:sig} but where the integrals are defined by Riemann and It\^o integration. Recall (e.g.\ \cite{baudoin}) the product rule \eqref{eq:shuffle} is replaced by
\begin{equation}\label{eq:qshuffle}
\langle u, \widehat S(\widetilde B) \rangle\langle v, \widehat S(\widetilde B) \rangle = \langle u \qshuffle v, \widehat S(\widetilde B) \rangle.
\end{equation}
The bracket in \eqref{eq:qshprod} then indexes quadratic variation, i.e.\ if $v_1, v_2 \in V$
\begin{equation}\label{eq:bracket_rho}
\langle [v_1, v_2] , \dif B_t \rangle = \varrho_{v_1,v_2} \dif t
\end{equation}
where $\varrho_{v_1,v_2} = \E[B^{v_1}_1 B^{v_2}_1]$ denotes the (constant) correlation of the components of $B$ in the directions $v_1, v_2$. Following the same logic of \eqref{eq:innerProd}, we define the inner product on $\widehat{\mathrm{Sh}}(\widetilde V)$
\begin{equation}\label{eq:innerquasi}
(u, v)_{\qshuffle} \coloneqq \mathbb E \langle u, \widehat S(\widetilde B)_{0,T} \rangle \langle v,  \widehat S(\widetilde B)_{0,T} \rangle = \langle u \qshuffle v , \mathbb E \widehat S(\widetilde B)_{0,T} \rangle .
\end{equation}
The introduction of time as a coordinate has the consequence of making the inner product degenerate, for example $\langle \zero - T \varnothing, \widehat S(\widetilde B)_{0,T}\rangle = 0$ a.s. The next proposition identifies a complement to the nullspace (cf.\ \cite[Theorem 3.9]{dupire2023functional} for a linear independence statement).
\begin{proposition}\label{eq:nullspace}
Let $N$ be the nullspace of $( \, \cdot \, , \, \cdot \, )_{\qshuffle}$ and denote $\widehat{\mathrm{Sh}}{}^\circ(\widetilde V)$ the vector subspace of $\widehat{\mathrm{Sh}}(\widetilde V)$ generated by words that do not end in a \emph{\zero}. Then $\widehat{\mathrm{Sh}}(\widetilde V) = N \oplus \widehat{\mathrm{Sh}}{}^\circ(\widetilde V)$.
\end{proposition}
\begin{proof}
Since $( \, \cdot \, , \, \cdot \, )_{\qshuffle}$ is just the $L^2$ inner product applied to the evaluation of a word on $\widehat S(\widetilde B)_{0,T}$, it suffices to show that (i) if $\rho \in \widehat{\mathrm{Sh}}{}^\circ(\widetilde V)$ such that $\mathbb E \langle \rho, \widehat S(\widetilde B)_{0,T} \rangle^2 = 0$ then $\rho = 0$, and (ii) that for any $w \in [d]^\bullet_\zero$ there exists a $\rho \in \widehat{\mathrm{Sh}}{}^\circ(\widetilde V)$ such that $\langle w\zero, \widehat S(\widetilde B)_{0,T} \rangle = \langle \rho, \widehat S(\widetilde B)_{0,T} \rangle$.

The claim (ii) follows inductively on the number of trailing $\zero$'s, since, by the (quasi-)shuffle relations, $\langle w\zero, \widehat S(\widetilde B)_{0,T} \rangle$ can be expressed as the product $\langle \zero, \widehat S(\widetilde B)_{0,T} \rangle \langle w, \widehat S(\widetilde B)_{0,T} \rangle$ minus terms with fewer trailing zeros, and $\langle \zero, \widehat S(\widetilde B)_{0,T} \rangle = T \in \mathbb R$. In order to prove (i), we first show that if the function 
\[
[0,T] \ni t \mapsto \langle \sigma, \widehat S(\widetilde B)_{0,t}\rangle
\]
vanishes identically on some $\sigma \in \widehat{\mathrm{Sh}}(\widetilde V)$, then $\sigma = 0$. Proceed by induction on the maximum length $n$ of a word in the linear combination for $\sigma$. If $n = 0$, $\langle \sigma, \widehat S(\widetilde B)_{0,t}\rangle \equiv \sigma \in \mathbb R$. For $n > 0$, write $\sigma = \lambda \varnothing + \sum_i \lambda_{\zero, i} w_{\zero,i}\zero + \sum_{\alpha = 1}^d \sum_i \lambda_{\alpha, i_\alpha} w_{\alpha,i}\alpha$. The process $\langle \sigma, \widehat S(\widetilde B)_{0,t}\rangle \equiv 0$ is then equal to a constant term plus a linear combination of Riemann and It\^o integrals, all of whose integrands must vanish identically, by Doob-Meyer, independence of components, and the fact that if $\int_0^t H_s \dif B^\alpha_s$ or $\int_0^t H_s \dif s$ vanishes for all $t \in [0,T]$ then $H$ must also vanish identically. Therefore $\lambda = 0$, $\sum_i \lambda_{\zero, i} w_{\zero,i}\zero = 0$ and $\sum_i \lambda_{\alpha, i_\alpha} w_{\alpha,i}\alpha = 0$ for all $\alpha$ and the induction is complete. Now, assume $\rho = \mu \varnothing + \sum_{\alpha = 1}^d \sum_j \lambda_{\alpha,j} w_j \alpha$ is as in the statement of (i) above. By the It\^o isometry
\[
0 = \mathbb E \langle \rho, \widehat S(\widetilde B)_{0,T} \rangle^2 = \sum_{\alpha = 1}^d \int_0^T \mathbb E\langle {\textstyle \sum_j} \lambda_{\alpha,j} w_j, \widehat S(\widetilde B)_{0,t} \rangle^2 \dif t
\]
and the integrand does not vanish identically, since the process $t \mapsto \langle {\textstyle \sum_j} \lambda_{\alpha,j} w_j, \widehat S(\widetilde B)_{0,t} \rangle$ does not, as just proved. 
\end{proof}

Write $[d]_\zero^\circ \subset [d]_\zero^\bullet$ for the subset of words that do not end in a $\zero$, so $\mathrm{span}[d]_\zero^\circ = \widehat{\mathrm{Sh}}{}^\circ(\widetilde V)$. Using It\^o integration has the benefit that, by independence and orthogonality of Wiener chaos, many pairs of words are already orthogonal. For $u, v \in [d]_\zero^\circ$ write $u \sim_\zero v$ if $u$ and $v$ are equal after stripping away zeros and leaving other letters in their place. Observe that (by the martingale property of It\^o integrals)
\begin{equation}\label{eq:eitosig}
\mathbb E \widehat S(\widetilde B)_{0,T} = \sum_{n = 0}^\infty \frac{T^n}{n!} \zero^n
\end{equation}
and therefore $(u,v)_{\qshuffle} = 0$ whenever $u \not \sim_\zero v$, since in this case $u \qshuffle v$ will be a linear combination of words none of which are of the form $\zero^n$. Moreover, if $u \sim_\zero v$, $(u,v)_{\qshuffle}$ will only depend on the number and position of the non-zero letters, i.e.\ it will equal the inner product of the words in the binary alphabet $\{\zero, \tone \!\}$ in which all non-zero entries in $u$ and $v$ are replaced with a $\tone$, for which we compute
\begin{equation}\label{eq:zeroOne}
(\zero^{i_1} \tone \zero^{i_2} \ldots \zero^{i_k} \tone, \zero^{j_1} \tone \zero^{j_2} \ldots \zero^{j_k} \tone)_{\qshuffle} = \frac{T^{i+j+k}}{(i+j+k)!} \prod_{r = 0}^k \binom{i_r + j_r}{i_r}, \qquad \sum_{l = 1}^k i_l = i, \ \sum_{l = 1}^k j_l = j.
\end{equation}
in which $i_l, j_l$ may be $0$. Define a linear order $<_\zero$ on each equivalence class mod $\sim_\zero$ as follows:
\begin{enumerate}
\item If $u \sim_\zero v$ and $\# u < \# v$ ($\Leftrightarrow |u| < |v| \Leftrightarrow$ $u$ has fewer zeros than $v$), then $u <_\zero v$;
\item If $u \sim_\zero v$ and $\# u = \# v$, then $u <_\zero v$ if $u$ is less than $v$ in the lexicographic order.
\end{enumerate}
The following definition is similar to \eqref{eq:p} but adapted to the current setting with $\zero$ playing a special role.
\begin{definition}[It\^o orthogonal basis]\label{def:ehat}
For $u \in [d]^\circ_\zero$ define $\widehat p_u$ by performing Gram-Schmidt orthogonalisation along its equivalence class mod $\sim_\zero$, according to its linear order:
\[
\widehat p_w = w - \sum_{\substack{v \sim_\zero w \\ v <_\zero w}} \frac{(w, \widehat p_v)_{\qshuffle}}{(\widehat p_v, \widehat p_v)_{\qshuffle}} \widehat p_v .
\]
\end{definition}
Each $\widehat p_w$ is a linear combination of other words that contain the same non-zero letters (as a subword, in fact), with the same number of zeros or fewer. 
We compute the terms of $\widehat p$ up to words of length $3$. We use letters $\alpha,\beta,\ldots$ to denote letters in $[d]$.
\begin{align*}
\widehat p_\varnothing &= \varnothing, \quad \widehat p_\zero = \zero, \quad \widehat p_{\zero \tone} = \zero \tone - \frac 12 \tone, \quad \widehat p_{\zero\zero\tone} = \zero\zero\tone - \frac 12 \zero\tone + \frac{1}{12} \tone \\
\widehat p_{\tone\tone} &= \tone\tone, \quad \widehat p_{\zero\tone\tone} = \zero\tone\tone - \frac 13 \tone\tone,\quad \widehat p_{\tone \zero \tone} = \tone\zero\tone + \frac 12 \zero\tone\tone - \frac 12 \tone\tone, \quad \widehat p_{\tone\tone\tone} = \tone\tone\tone .
\end{align*}

While we have described the orthogonal basis $\{\widehat p_w \mid w \in [d]^\circ_\zero\}$ in terms of a basis on $V$, it comes from an intrinsic orthogonalisation map. Indeed, we have decomposed
\begin{equation}\label{eq:directsum}
\widehat{\mathrm{Sh}}{}^\circ(\widetilde V) = \bigoplus_{k = 0}^\infty W^k, \qquad W^k = \bigoplus_{i = 0}^\infty \underbrace{\bigoplus_{i_1 + \ldots + i_k = i} W^k_{i_1\ldots i_k}}_{|\, \cdot \,| = k + 2i, \ \# = k + i}
\end{equation}
with $W^k$ the space generated by all words with $k$ non-zero letters (ending in a non-zero letter) and $W^k_{i_1\ldots i_k}$ its subspace of generated by words of the form $\zero^{i_1} v_1 \zero^{i_2} \ldots \zero^{i_k} v_k$ with $v_l \in V$ and the multiplicities $i_l$ possibly zero. The first direct sum is orthogonal, since $W^k$ maps to the $k^\text{th}$ Wiener chaos, but the other two are not. For each fixed $k$, the blocks $W^k_{i_1\ldots i_k}$ are linearly ordered according to the order on sequences of integers given by $i_1 \dots i_k < i_1' \dots i_k'$ if $\sum_l i_l < \sum_l i_l'$ or if $\sum_l i_l = \sum_l i_l'$ and for some $h \leq k$, $i_l = i_l'$ for $l < h$ and $i_h < i_h'$. We then perform block-orthogonalisation $\widehat p$ according to this order on the blocks $W^k_{i_1\ldots i_k}$: $\widehat p_w$ is just the orthogonal projection of $w$ onto the direct complement of all preceding blocks in $W^k$. Denote the direct sum of these maps $\widehat p^V \colon \widehat{\mathrm{Sh}}{}^\circ(\widetilde V) \to \widehat{\mathrm{Sh}}{}^\circ(\widetilde V)$. A basis of $V$ additionally yields an orthogonal basis of each block $W^k_{i_1\ldots i_k}$, and orthogonality is preserved by $\widehat p$, yielding an orthogonal basis. Notice that only a basis of $V$, not a frame as in the more general setting \eqref{eq:p}, is necessary.

\begin{theorem}[Natural orthogonalisation of Wiener functionals]\label{thm:natural_time}
With $\mathcal F$ the sigma-algebra generated by $(B_t)_{t \in [0,T]}$, $Y \in L^2(\mathcal F)$ can be represented as the series with pairwise orthogonal summands, convergent in $L^2$
\[
Y = \sum_{w \in [d]^\circ_{\emph{\zero}}} \frac{\mathbb E[Y \langle \widehat p_w, \widehat S(\widetilde B)_{0,T}\rangle]}{(\widehat p_w, \widehat p_w)_{\qshuffle}} \langle \widehat p_w, \widehat S(\widetilde B)_{0,T} \rangle.
\]
Moreover, $\widehat p$ is a natural orthogonalisation according to the functor $V \mapsto \widehat{\mathrm{Sh}}{}^\circ(\widetilde V)$, i.e.\ for any injective isometry of Euclidean spaces $\varphi \colon U \to V$, the diagram
\[
\begin{tikzcd}
    \widehat{\mathrm{Sh}}{}^\circ(\widetilde U) \arrow[r, "\widehat p^U"] \arrow[d, "\widehat{\mathrm{Sh}}{}^\circ(\varphi)"] & \widehat{\mathrm{Sh}}{}^\circ(\widetilde U) \arrow[d, "\widehat{\mathrm{Sh}}{}^\circ(\varphi)"] \\
    \widehat{\mathrm{Sh}}{}^\circ(\widetilde V) \arrow[r, "\widehat p^V"] & \widehat{\mathrm{Sh}}{}^\circ(\widetilde V)
\end{tikzcd}
\]
(with $\widehat{\mathrm{Sh}}{}^\circ(\varphi) \coloneqq T(\one \oplus \varphi)|_{\widehat{\mathrm{Sh}}{}^\circ(\widetilde U)}$) commutes.
\end{theorem}
\begin{proof}
The first part follows directly from \Cref{cor:series} (or rather the same proof; since \Cref{thm:density} or \cite{lpdensity} apply to time-augmented Brownian motion), \eqref{eq:nullspace} and the preceding discussion. 
Let $\Phi \coloneqq \widehat{\mathrm{Sh}}{}^\circ(\varphi)$. It follows from the preceding discussion that the decomposition into blocks $W^k_{i_1\ldots i_k}$ is preserved by $\Phi$, and moreover $\Phi$ acts on $W^k_{i_1\ldots i_k} = W^k_{i_1\ldots i_k}(U) \cong U^{\otimes k}$ by $\varphi^{\otimes k}$. Since $\widehat p$ operates on each block $W^k$ independently, we have reduced the problem to proving, for each $i,k \geq 0$, commutation of the squares
\[
\begin{tikzcd}
\bigoplus_{i_1 + \ldots + i_k = i} U^{\otimes k} \arrow[r, "\widehat p^U"] \arrow[d,"\varphi^{\otimes k}"] & \bigoplus_{i_1 + \ldots + i_k = i} U^{\otimes k} \arrow[d,"\varphi^{\otimes k}"] \\
\bigoplus_{i_1 + \ldots + i_k = i} V^{\otimes k} \arrow[r, "\widehat p^V"] & \bigoplus_{i_1 + \ldots + i_k = i} V^{\otimes k}.
\end{tikzcd}
\]
This follows from the fact that $\varphi$ preserves the bracket in \eqref{eq:qshprod}, i.e.\ $[\varphi(u),\varphi(u')] = [u,u']$ and therefore preserves inner products in $W^k$:
\[
(\zero^{i_1} \varphi(u_1) \zero^{i_2} \ldots \zero^{i_k} \varphi(u_k), \zero^{j_1} \varphi(u_1') \zero^{j_2} \ldots \zero^{j_k} \varphi(u_k')) = (\zero^{i_1} u_1 \zero^{i_2} \ldots \zero^{i_k} u_k, \zero^{j_1} u_1' \zero^{j_2} \ldots \zero^{j_k} u_k').
\]
The whole Gram-Schmidt process \Cref{def:ehat} is thus preserved and the claim follows.
\end{proof}

\begin{remark}[Stratonovich orthogonalisation]\label{rem:strat}
A similar orthogonalisation map can be obtained when working with the Stratonovich signatures. We only sketch this briefly, since it is easier to interpret terms in the It\^o setting. Recall that the Hu-Meyer formulae \cite{humeyer} (later generalised to the Hoffman's exponential \cite{hoff} in a different context) establish natural algebra isomorphisms
\[
\exp^V \colon \mathrm{Sh}(\widetilde V) \to \widehat{\mathrm{Sh}}(\widetilde V), \qquad \log^V \colon \widehat{\mathrm{Sh}}(\widetilde V) \to \mathrm{Sh}(\widetilde V)
\]
inverses of each other. A quick inspection of these maps reveals that they respect the direct sum of \Cref{eq:nullspace}, since the bracket does not act on the $\zero$ coordinate. We use $(\, \cdot \, , \, \cdot \,)_\shuffle$ to denote the inner product on $\mathrm{Sh}(\widetilde V)$ in which evaluation is against the Stratonovich signature $S(\widetilde W)_{0,T}$. Then
\[
(u, v)_\shuffle = \langle u \shuffle v , S(\widetilde B)_{0,T} \rangle = \langle \exp^V(u) \shuffle \exp^V(v) , \widehat S(\widetilde B)_{0,T} \rangle = (\exp^V(u), \exp^V(v))_{\qshuffle},
\]
i.e.\ $\exp^V$ (and $\log^V$) define isometries. Therefore, setting
\[
q^V \colon \mathrm{Sh}(\widetilde V) \to \mathrm{Sh}(\widetilde V), \qquad q^V = \exp^V \circ \widehat p^V \circ \log^V
\]
all properties of $\widehat p$ carry over to the Stratonovich-shuffle case.
\end{remark}

\begin{remark}[Hermite and Legendre]
We observe that $\widehat p$ embeds both Hermite and Legendre polynomials: the former correspond to $n!\widehat p_{\tone^n}$, the latter (shifted to the interval $[0,T]$) to $n!\widehat p_{\zero^n}$.
\end{remark}

\begin{remark}[Comparison with Wiener chaos]
Since orthogonalisation map $\widehat p$ works largely because of orthogonality of Wiener chaos, it is interesting to probe the link with this decomposition further. Recall that, letting $\mathcal H = L^2([0,T], \mathbb R^d) = L^2([0,T] \times [d])$ there exists an isometry given by multiple Wiener integration (see e.g.\ \cite{NP12})
\begin{equation}\label{eq:wienerchaos}
\bigoplus_{n = 0}^\infty L^2_\mathrm{sym}(\underbrace{([0,T] \times [d]) \times \cdots \times ([0,T] \times [d])}_{n}) = \bigoplus_{n = 0}^\infty \mathcal H^{\odot n} \cong L^2(\mathcal F)
\end{equation}
where the subscript $\mathrm{sym}$ refers to symmetry of the functions in the $n$ factors. A straightforward way of obtaining an orthogonal basis of $L^2(\mathcal F)$ is that of fixing one of $L^2[0,T]$ (Fourier, Legendre, etc.), lifting it to one $\{\varphi_m\}_{m \in \mathbb N}$ of $L^2([0,T], \mathbb R^d)$, and of taking symmetric powers of it. This yields a representation
\begin{equation}\label{eq:nm-basis}
Y = \sum_{n = 0}^\infty \sum_{m = 0}^\infty\sum_{m_1 + \ldots + m_n = m} \lambda_{m_1,\ldots,m_n}^n \mathcal \int_{[0,T]^n} \varphi_{m_1,\alpha_1}(t_1) \cdots \varphi_{m_n,\alpha_n}(t_n) \dif W_{t_1}^{\alpha_1} \cdots \dif W_{t_n}^{\alpha_n} .
\end{equation}
with implicit summation on each $\alpha_l$ and coefficients $\lambda$ symmetric in the lower indices. This basis however depends on two parameters $n$ and $m$, and in order to obtain a one-parameter basis, one must make a choice of how to trade off truncation in $n$ with truncation in $m$. Notice, moreover, that quantities of interest in differential equations, such as the Lévy areas $\frac 12 \int_0^T (B^\alpha \dif B^\beta - B^\beta \dif B^\alpha)$, which are not the Wiener integral of symmetric kernels in the sense above, have infinite expansion w.r.t.\ \eqref{eq:nm-basis}. \Cref{def:ehat} is adapted to Wiener chaos but sidesteps these issues, thanks to the time-ordered nature of the signature.
\end{remark}

Pushing similar considerations a little further we are able to obtain a quick proof of \Cref{thm:density} in the case of time-augmented Brownian motion and $p = 2$. Recall that, even though $\mathcal F = \mathcal G$ by \Cref{rem:FG}, the proposition does not show that linear functions on $S(B)_{0,T}$ (without time-augmentation) is dense in $L^2(\mathcal F)$, and \Cref{thm:density} is needed for this.

\begin{proposition}[$L^2$-density of signature functionals]
\label{thm:density_brownian}
Linear functions on $S(\widetilde B)_{0,T}$ are dense in $L^2(\mathcal F)$.
\end{proposition}
\begin{proof}
By the It\^o-Stratonovich conversion formulae for iterated integrals of Brownian motion with time \cite{humeyer}, we may equivalently prove the statement for linear functions on $\widehat S(\overline B)_{0,T}$. Since polynomials are dense in $L^2[0,T]$, by \eqref{eq:wienerchaos} it suffices to show that for any $n, m_1,\ldots,m_n \in \mathbb N$ and any $\alpha_1,\ldots,\alpha_n \in [d]$ the multiple Wiener integral
\[
\int_{[0,T]^n} t_1^{m_1} \cdots t_n^{m_n} \dif W_{t_1}^{\alpha_1} \cdots \dif W_{t_n}^{\alpha_n}
\]
can be expressed as a linear function of $\widehat S(\widetilde B)_{0,T}$. We express it as, up to a factor of $m_1! \cdots m_n!$
\begin{align*}
&\int_{[0,T]^n} \bigg(\int_{\Delta^{m_1}[0,t_1]} \dif s_{1,1} \cdots \dif s_{1,m_1} \bigg) \cdots \bigg(\int_{\Delta^{m_n}[0,t_n]} \dif s_{n,1} \cdots \dif s_{n,m_n} \bigg) \dif W_{t_1}^{\alpha_1} \cdots \dif W_{t_n}^{\alpha_n} \\
={}&\int_{[0,T]^n \ltimes (\Delta^{m_1}[0,t_1] \times \cdots \times \Delta^{m_n}[0,t_n])} \dif s_{1,1} \cdots \dif s_{1,m_1}  \cdots \dif s_{n,1} \cdots \dif s_{n,m_n} \dif W_{t_1}^{\alpha_1} \cdots \dif W_{t_n}^{\alpha_n}
\end{align*}
where for a family of sets $(B_a)_{a \in A}$ we denote $A \ltimes B_a \coloneqq \{(a,b) \mid a \in A, b \in B_a\}$. Writing $[0,T]^n = \bigsqcup_{\sigma \in \mathfrak S_n} \sigma_*\Delta^n[0,T]$, we reduce the problem of expressing (up to reordering) $\Delta^n[0,T] \ltimes (\Delta^{m_1}[0,t_1] \times \cdots \times \Delta^{m_n}[0,t_n])$ as a disjoint union of simplices over $[0,T]$. By treating the two-factor case
\begin{align*}
&\Delta^i[0,a] \times \Delta^j[0,b]\\
={}&\bigsqcup_{k = 1}^j \{0 < u_1 < \ldots < u_i < a\} \times \{ 0 < u_{i+1} < \ldots < u_{i+k} < a < u_{i+k+1} < \ldots < u_{i+j} < b\} \\
={}&\bigsqcup_{k = 1}^j\bigsqcup_{\sigma \in \mathrm{Sh}(i,k)} \{0 < u_{\sigma(1)} < \ldots < u_{\sigma(i+k)} < a < u_{\sigma(i+k+1)} < \ldots < u_{\sigma(i+j)} < b\}
\end{align*}
we reduce, by induction, to sets of the form
\[
\Delta^n[0,T] \ltimes (\Delta^{i_1}[0,t_1] \times \Delta^{i_2}[t_1,t_2] \times \cdots \times \Delta^{i_n}[t_{n-1}, t_n]) = \Delta^{n + (i_1 + \ldots + i_n)}[0,T]
\]
and the proof is complete.
\end{proof}

\subsection{Numerical experiments}\label{subsec:numerics}

In this subsection we perform numerical experiments on the orthogonalised It\^o signature developed in the previous. The first necessary step is to obtain the Gram-Schmidt-orthogonalised It\^o-signature features, up to a given inhomogeneous degree (that is with the drift coordinate counting double). We use the package \texttt{Signax} \cite{signax} for computing signatures and \cite{reizenstein} for algebraic manipulation on the tensor algebra. Since \texttt{Signax} (as well as all other packages that compute signatures), when called on Brownian motion, compute the Stratonovich signature, we must first convert to It\^o form. This is done symbolically by implementing the Hoffman logarithm (see \Cref{rem:strat}). Note that, on finitely discretised time series, this is not equivalent to computing It\^o integrals as left-endpoint Riemann sums, but they are equivalent in the limit of vanishing mesh size. In order to perform the Gram-Schmidt orthogonalisation, the inner product $(u, v)_{\qshuffle}$ \eqref{eq:innerquasi} is computed on binary words as in \eqref{eq:zeroOne}. The Gram-Schmidt basis \eqref{def:ehat} is then calculated, still in the binary case, and only then \say{mapped on} to the case of general $d$ as described in \Cref{subsec:ito}. This is an expression of the naturality of the orthogonal basis \Cref{thm:natural_time} and avoids the explicit computation of the (quasi-)shuffle product or Monte Carlo evaluation which would normally be necessary for \eqref{eq:shuffle}, significantly speeding up computation of the inner product \eqref{eq:qshuffle}. For further details, we refer to our implementation \cite{orthsig} and its documentation. We are able to check that the orthogonalisation achieves the required goal, see \Cref{fig:heatmap}.
\begin{figure}[h!]
    \centering
    \resizebox{\textwidth}{!}{%
        \mbox{%
            \includegraphics[height=3cm]{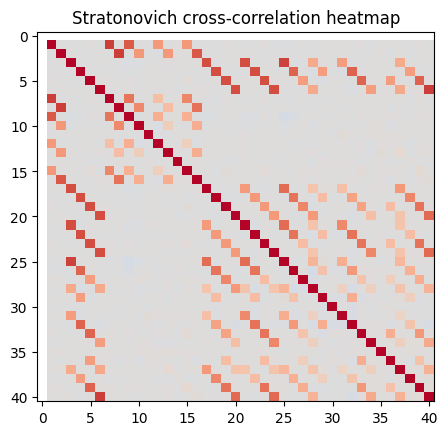}%
            \hspace{0.5em}%
            \includegraphics[height=3cm]{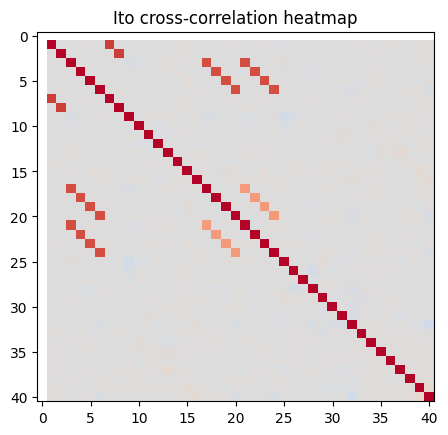}%
            \hspace{0.5em}%
            \includegraphics[height=3cm]{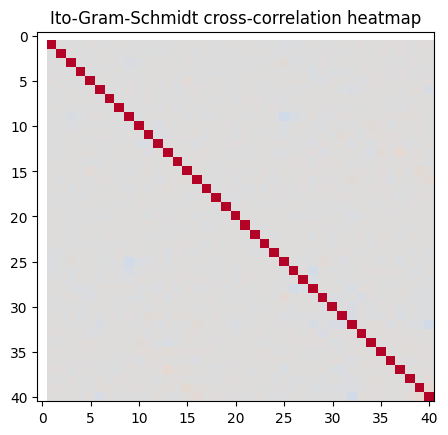}%
            \hspace{0.5em}%
            \includegraphics[height=3cm]{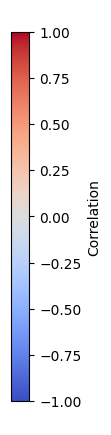}%
        }%
    }
    \caption{Comparison of correlation heatmaps for flattened signatures, on the subspace making $( \, \cdot \,, \, \cdot \,)_{\qshuffle}$ is non-degenerate \eqref{eq:nullspace}, computed over $100$k $2$-dimensional time-augmented sample Brownian paths, with $T = 1$ and $1$k grid points. The Stratonovich signature features are far from orthogonal, the It\^o ones are much sparser but still not orthogonal (because of residual correlations inside each Wiener chaos), and finally their Gram-Schmidt orthogonalisation is verified to be fully orthogonal (modulo numerical errors). This and similar checks can be performed with the notebook \cite[\texttt{orth\_checks.ipynb}]{orthsig}.}\label{fig:heatmap}
\end{figure}

We consider two related but distinct tasks. In \emph{functional expansion}, the function on paths is known and the task is to approximate it as a signature expansion. This can often be done independently of data; for example if the function is an SDE \eqref{eq:RDEs}, one can use a numerical method like \eqref{eq:taylor} iterated over many intervals. In \emph{functional regression}, the function on paths is not known or hard to expand analytically, rather we have access to i.i.d.\ input-output pairs, but the goal is similar. We will use orthogonal polynomials on Wiener space in a similar way for both, but the methodologies we compare with for each are distinct.

We consider a linear SDE
\begin{equation}
\dif Y^k = A^k_{\alpha i} Y^i \dif W^\alpha, \quad Y_0 = y_0,
\end{equation}
which is one of the rare cases for which the stochastic Taylor expansion \eqref{eq:taylor} converges on a single interval \cite[\S 4.2]{LCL}. We compare the performance of the Taylor scheme with that of evaluating the truncated series \Cref{thm:natural_time} on a sample of Brownian paths, see \Cref{fig:sde_L2_R2}. We observe convergence of both methods, with the Taylor method lagging behind the $L^2$ for low degrees, and catching up to the orthogonal expansion at higher degrees, especially for lower sample sizes (which only affects the orthogonal expansion). For a non-linear SDE whose solution lies in $L^2$, \Cref{thm:natural_time} still applies but examples such as \eqref{eq:legendre} show that one cannot expect convergence of the Taylor scheme on one interval, and thus iterating the method would be necessary; this however is not a signature expansion as usually intended in the machine learning literature.

\begin{figure}[htbp]
    \centering

    \begin{minipage}[t]{0.48\textwidth}
        \centering
        \includegraphics[width=\linewidth]{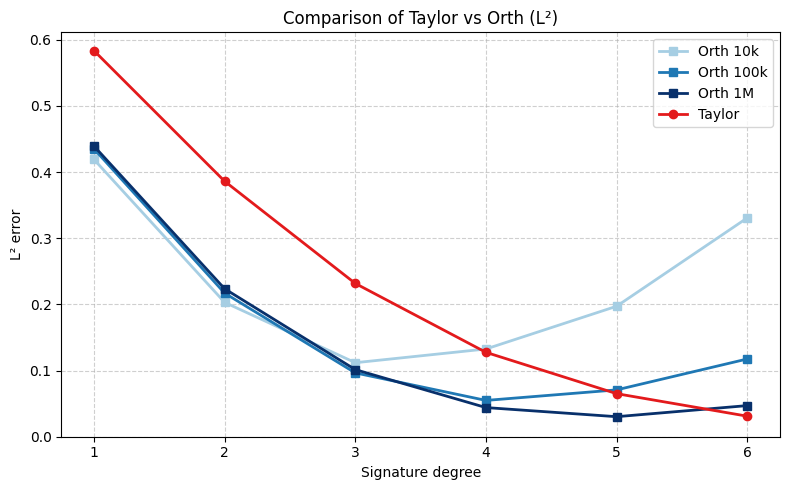}
    \end{minipage}
    \hfill
    \begin{minipage}[t]{0.48\textwidth}
        \centering
        \includegraphics[width=\linewidth]{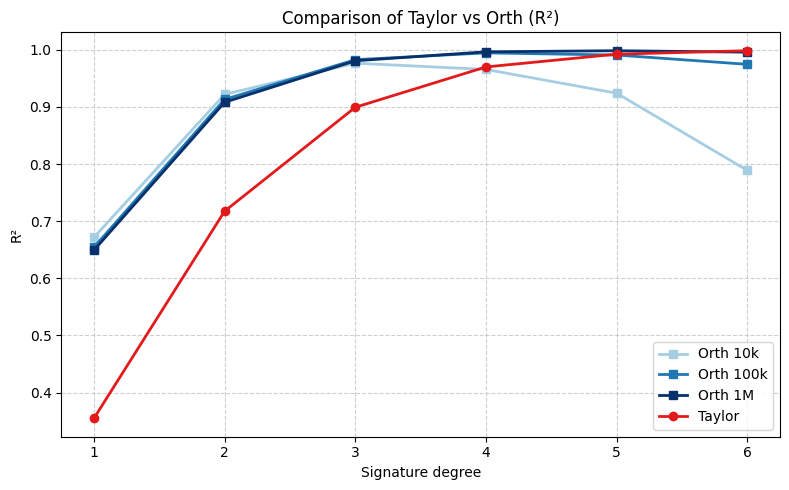}
    \end{minipage}

    \caption{Comparison of Taylor and Orth models: out-of-sample L$^2$ error (left) and coefficient of determination (R$^2$, right) across different dataset sizes. The Brownian motion is taken to have dimension $d = 2$, $Y$ is scalar, and errors/R$^2$ are averaged over $10$ random choices of the matrix $A$ normalised to have Euclidean norm $1$.}
    \label{fig:sde_L2_R2}
\end{figure}

In the next example we consider the a scalar Black Scholes model, and learn the two functions: an at-the-money call
option and a lookback option, i.e.\ just the maximum of the path (a difficult case for functional expansions
\cite[Example 2]{dupire19}), see \Cref{fig:sde_R2_BS}. Note that unlike in the literature on applications of signatures to finance (e.g.\ \cite{arribas2018derivatives}), we are learning/expanding the payoff in terms of the underlying Brownian motion, not the price path. For the latter, we would need to orthogonalise the signature of paths drawn from geometric Brownian motion. We observe convergence of both methods, both in the sample size and truncation level, although the Monte Carlo estimator for the orthogonal expansion, i.e.\ the truncated series in \Cref{thm:natural_time}, appears to converge more slowly than the OLS estimator \eqref{eq:OLS} in terms of the sample size, especially at higher degree. Notice that these are two distinct estimators for the same quantity, $\Pi_N Y$. In particular, for the orthogonal expansion estimator the difference between in- and out-of sample error is less marked, since it is not the solution to a data-dependent optimisation problem.

\begin{figure}[htbp]
    \centering

    \begin{minipage}[t]{0.48\textwidth}
        \centering
        \includegraphics[width=\linewidth]{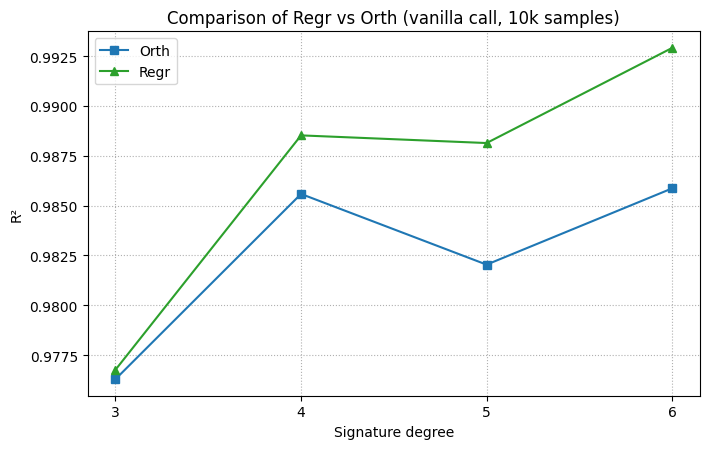}
    \end{minipage}
    \hfill
    \begin{minipage}[t]{0.48\textwidth}
        \centering
        \includegraphics[width=\linewidth]{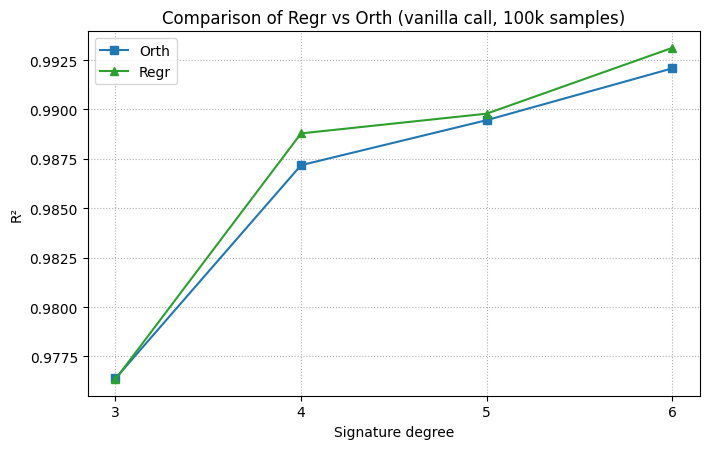}
    \end{minipage}

        \begin{minipage}[t]{0.48\textwidth}
        \centering
        \includegraphics[width=\linewidth]{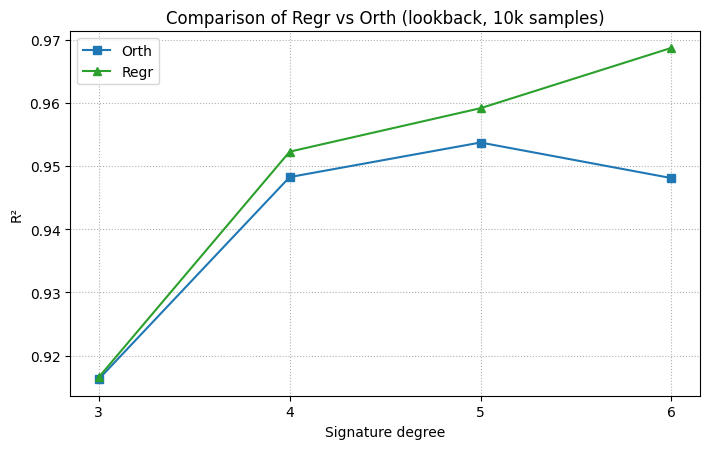}
    \end{minipage}
    \hfill
    \begin{minipage}[t]{0.48\textwidth}
        \centering
        \includegraphics[width=\linewidth]{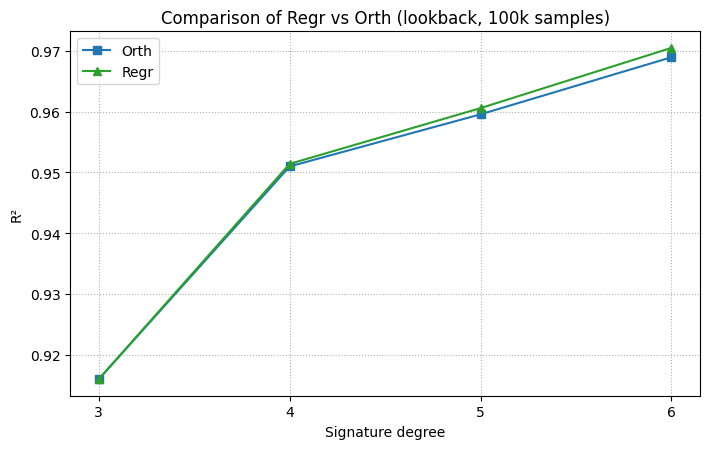}
    \end{minipage}

    \caption{Comparison of coefficients of determination for OLS regression (Regr) on the truncated signature (with non-orthogonal coordinates, out-of-sample) and orthogonal signature expansion (Orth). The scalar Black Scholes model has parameters $\sigma = 0.2$, $\mu = 0$, $S_0 = 1$, and the call option is struck at $K = 1$ at time $T = 1$. We observed worse performance for both models for OTM options, and for ITM options Orth was performing worse than Regr.}
    \label{fig:sde_R2_BS}
\end{figure}

Evaluating the expansion \Cref{thm:natural_time} is linear in both sample size $M$ and number of features $D = \frac{d^{N+1} - 1}{d-1}$, while performing the matrix inversion in \eqref{eq:OLS} has complexity which is cubic in either $K$ or $M$ depending on whether it is evaluated in the primal or dual formulation. In practice, the bottleneck lies in the overhead costs of computing the signatures of large samples of paths. Also, when sampling from a fixed measure, the data matrix could be computed once and for all and re-used across many different problems. Even when performing regression, however, using the orthogonal basis rather than the ordinary basis has advantages. Orthogonality means coefficients can be learnt incrementally in degree, unlike for non-orthogonal features. It can also improve conditioning, see \cite{tian} for a comparison in the case of polynomials (ill conditioning of Vandermonde matrices is a well known phenomenon \cite{pan16}). Finally, expanding a whole collection of payoffs as in \Cref{thm:natural_time} has the benefit that covariances can be more efficiently estimated by truncating the series
\begin{align*}
\mathbb E[Y_1 Y_2] &= \sum_{w \in [d]^\circ_{\zero}} \frac{\mathbb E[Y_1 \langle \widehat p_w, \widehat S(\widetilde B)_{0,T}\rangle] \mathbb E[Y_2 \langle \widehat p_w, \widehat S(\widetilde B)_{0,T}\rangle]}{(\widehat p_w, \widehat p_w)_{\qshuffle}^2} \langle \widehat p_w, \widehat S(\widetilde B)_{0,T} \rangle^2 \\
&= \sum_{w \in [d]^\circ_{\zero}} \frac{\mathbb E[Y_1 \langle \widehat p_w, \widehat S(\widetilde B)_{0,T}\rangle] \mathbb E[Y_2 \langle \widehat p_w, \widehat S(\widetilde B)_{0,T}\rangle]}{(\widehat p_w, \widehat p_w)_{\qshuffle}}
\end{align*}
This could have applications, for example, in stochastic portfolio theory \cite{CM24}.

\bibliographystyle{plain}
\bibliography{noncommutative_orthog_polys}

\appendix
\lstset{numbers=left, numberstyle=\tiny, stepnumber=2, numbersep=5pt, basicstyle=\small, tabsize=2,
xleftmargin=.1\textwidth, language=Mathematica}
\section{Linear systems}\label{ap:code}
We explain the procedure used to study the naturality of the block-orthogonalization map in the shuffle algebra.
The full code listings can be found on GitHub \cite{orth}.
Starting from Fawcett's formula in \eqref{eq:fawcett} and the inner product on words in \eqref{eq:innerProd}, we
implemented the expected signature and shuffle product in \texttt{Wolfram 14}, in order to enable symbolic computations.
For simplicity we take \(T=1\).
The idea with this approach is that since the natural basis is not orthogonal under \((\cdot,\cdot)\), inversion of the
Gram matrix yielding the projection at degree 5 is computationally intensive.

We begin by creating an abstract symbol \(w\) that acts as a placeholder for a word, and declaring it to obey a certain
set of rules with respect to the CircleTimes (\(\otimes\)) operator, which acts as an external tensor product.

\begin{lstlisting}[firstline=57, lastline=65, caption={Tensor product of words}]
ClearAll[CircleTimes]
SetAttributes[CircleTimes, {Flat, OneIdentity}]
CircleTimes[x___, y_Plus, z___] := CircleTimes[x, #, z]& /@ y
CircleTimes[x___, -a_, y___] := -CircleTimes[x, a, y]
CircleTimes[x___, a_, y_] := a CircleTimes[x, y] /; FreeQ[a, w] 
CircleTimes[x_, a_, y___] := a CircleTimes[x, y] /; FreeQ[a, w]
CircleTimes[x___, a_ y_, z_] := a CircleTimes[x, y, z] /; FreeQ[a, w]
CircleTimes[x_, a_ y_, z___] := a CircleTimes[x, y, z] /; FreeQ[a, w]
CircleTimes[p_ x_]:=p CircleTimes[x]/;FreeQ[p, w]
\end{lstlisting}

We then implement the concatenation operator Conc, which maps \(T(V)\otimes T(V)\to T(V)\) \emph{linearly}:

\begin{lstlisting}[firstline=67, lastline=73, caption={Concatenation of words}]
ClearAll[Conc]
SetAttributes[Conc, OneIdentity]
Conc[0] = 0;
Conc[p_ x_] := p Conc[x] /; FreeQ[p, w]
Conc[x_Plus] := Plus[Conc/@x]
Conc[CircleTimes[x__w]] := w@@Delete[0] /@ Level[{x},1]
Conc[CircleTimes[x_]]:=x
\end{lstlisting}

Finally, the shuffle product on w[...] symbols is implemented as a linear map \(T(V)\otimes T(V)\to T(V)\)
recursively:

\begin{lstlisting}[firstline=75, lastline=83, caption={Shuffle product of words}, breaklines]
ClearAll[Shuf]
SetAttributes[Shuf, OneIdentity]
Shuf[x_w]:=x
Shuf[p_ x_]:=p Shuf[x]/;FreeQ[p,w]
Shuf[x_Plus]:=Plus[Shuf/@x]
Shuf[x_w\[CircleTimes]w[]]:=x
Shuf[w[]\[CircleTimes]x_w]:=x
Shuf[w[u__]\[CircleTimes]w[v__]]:=Conc[Shuf[w[u]\[CircleTimes]w@@Drop[{v},-1]]\[CircleTimes]w[{v}[[-1]]]]+Conc[Shuf[w@@Drop[{u},-1]\[CircleTimes]w[v]]\[CircleTimes]w[{u}[[-1]]]]
Shuf[x__w\[CircleTimes]y_w]:=Shuf[Shuf[CircleTimes[x]]\[CircleTimes]y]
\end{lstlisting}
The last relation enforces associativity.

Next, we implement \eqref{eq:fawcett} on pure words w[...] and extend by linearity.

\begin{lstlisting}[firstline=161, lastline=169, caption={Expected signature of time-augmented Brownian motion}]
ClearAll[ESig]
ESig[w[]]=1;
ESig[w[x__]]:=Block[{z=Cases[{x},0],l=DeleteCases[{x},0]}, 
	With[{n=Length[l], m=Length[z]},
		If[EvenQ[n],2^(-n/2)/(n/2+m)! Times@@\[Delta]@@@Partition[l,2],0]
	]
]
ESig[x_Plus]:=Plus[ESig/@x]
ESig[p_ x_]:=p ESig[x]/;FreeQ[p,w]
\end{lstlisting}
Here, the symbol \textbackslash[Delta] (\(\delta\)) is subject to the symmetry rule
\begin{lstlisting}[firstline=158, lastline=159]
ClearAll[\[Delta]]
\[Delta]/:\[Delta][b_, a_] := \[Delta][a, b] /; b>a
\end{lstlisting}

Finally this induces an inner product on words by \eqref{eq:innerProd}:

\begin{lstlisting}[firstline=171, lastline=176, caption={Inner product on the shuffle algebra}]
ClearAll[ip]
ip[x_w,y_w]:=ESig[Shuf[x\[CircleTimes]y]]
ip[x_Plus,y_]:=Plus[ip[#,y]&/@x]
ip[x_,y_Plus]:=Plus[ip[x,#]&/@y]
ip[p_ x_,y_]:=p ip[x,y]/;FreeQ[p,w]
ip[x_,p_ y_]:=p ip[x,y]/;FreeQ[p,w]
\end{lstlisting}

In the next step, we use these functions to build the Ansatz in \Cref{sec:nonexist}, where we also include all crossing
partitions.
\begin{lstlisting}[breaklines,caption={Generation of the Ansatz for the orthogonalization map}]
Pairings[0] = 0; Parings[1] = 0; 
Pairings[n_ /; n > 1] := Select[Combinatorica`SetPartitions[n], AllTrue[Length[#] <= 2 &]]
vars[n_ /; n > 0] := Join[Array[a, Length[Pairings[n]] - 1], {1}]
Ansatz[n_ /; n > 0] := vars[n] . Conc /@ Map[CircleTimes @@ # &, Map[If[Length[#] == 1, w @@ #, \[Delta] @@ # w[]] &, Pairings[n], {2}], {1}]
\end{lstlisting}

Thus, the command Ansatz[3] generates the following output:
\begin{lstlisting}[breaklines, mathescape]
w[1, 2, 3] + a[2]w[3]$\delta$[1, 2] + a[3]w[2]$\delta$[1, 3] + a[1]w[1]$\delta$[2, 3]
\end{lstlisting}


We then generate the orthogonality relations:
\begin{lstlisting}[breaklines, mathescape]
Lower[n_ /; EvenQ[n]] := With[{pol = Ansatz[n]}, Table[ESig[Shuf[pol\[CircleTimes]w @@ (n + Range[2 k])]], {k, 0, n/2 - 1}]]
Lower[n_ /; OddQ[n]] := With[{pol = Ansatz[n]}, Table[ESig[Shuf[pol\[CircleTimes]w @@ (n + Range[2 k + 1])]], {k, 0, (n - 1)/2 - 1}]]
$\delta$vars[n_] := DeleteDuplicates[Cases[Lower[n], _$\delta$, Infinity]];
\end{lstlisting}

The command Lower[n] generates a systems of equations that have to be solved for in the \(a\) variables.
Since these should hold for any choice of letters and not just w[1,2,3,4], we use the SolveAlways instruction.
\begin{lstlisting}
Sol[n_] := SolveAlways[# == 0 & /@ Lower[n], \[Delta]vars[n]]
\end{lstlisting}

For example, running Sol[3] yields:
\begin{lstlisting}
{{a[1] -> -(1/4), a[2] -> -(1/4), a[3] -> 0}}
\end{lstlisting}
which is precisely the solution in \eqref{eq:diagEx}.
Note that a[3], i.e., the variable correspoding to the pairing \(\{\{1,3\},\{2\}\}\) is set to zero.

In order to check if the system has a solution, we may look at the matrix corresponding to the system of linear
equations generated by Lower[n], say \(A\), and check if the equality \(\rank(A)=\rank([A\mid b])\) holds, where \(b\) denotes
the vector of constant terms.
If the equality does not hold, it means that the system is inconsistent and therefore has no solution. In Wolfram this
is implemented by the following functions:
\begin{lstlisting}[breaklines, mathescape]
ClearAll[PreCoefs, Coefs, AugCoefs]
PreCoefs[n_] := PreCoefs[n] = CoefficientArrays[# == 0 & /@ Flatten[Values@CoefficientRules[#, $\delta$vars[n]] & /@ Lower[n]], Drop[vars[n], -1]] // Normal
Coefs[n_] := Coefs[n] = PreCoefs[n][[2]]
AugCoefs[n_] := Transpose[Insert[Transpose[Coefs[n]], -PreCoefs[n][[1]], -1]]
\end{lstlisting}
giving \(A=\text{Coefs[n]}\) and \([A\mid b]=\text{AugCoefs[n]}\) for a given word length \(n\).

 For the case of interest, namely \(n=5\), the matrix
has \(\rank(A)=25\) and \(\rank([A\mid b])=26\).
A subset of contradicting equations may be obtained by inspecting this matrix or by computing a certificate by computing
a basis of its column null space by solving the system \(y^\top A=0\), and then looking for a vector such that \(y^\top
b\neq 0\). The nonzero
entries of such a vector determine an inconsistent set of equations. This is essentially a version of Farkas' lemma.

In the same way it may be checked that for \(n=6,7\) the ranks of the corresponding \(A\) are \(75\) and \(231\),
respectively, while the ranks of the augmented matrices are \(76\) and \(232\).
\end{document}